\documentclass[11pt,english]{smfart}

\usepackage{amsmath,leftidx,amsthm}
\usepackage[all]{xy}
\usepackage{xspace,amsthm}
\usepackage[psamsfonts]{amssymb}
\usepackage[latin1]{inputenc}
\usepackage{graphicx,color}
\usepackage{hyperref,fancyhdr,dutchcal,appendix}
\usepackage{xcolor} 
\usepackage{url}

\newtheorem{theorem}{\bf Theorem}[section]
\newtheorem{proposition}[theorem]{\bf Proposition}
\newtheorem{definition}[theorem]{\bf Definition}
\newtheorem{Theorem}{\bf Theorem}

\newtheorem{lemma}[theorem]{\bf Lemma}
\newtheorem{corollary}[theorem]{\bf Corollary}
\newtheorem{conjecture}[theorem]{\bf Conjecture}

\newtheorem{example}[theorem]{\bf Example}

\newtheorem{remark}[theorem]{\bf Remark}

\def\C{{\mathbb C}}
\def\N{{\mathbb N}}
\def\R{{\mathbb R}}

\def\B{\mathbb{B}}

\def\p{\mathbb{P}}

\def\supp{\textup{supp}}

\def\Per{\textup{Per}}

\def\and{{\quad\text{and}\quad}}

\title{The geometric dynamical Northcott and Bogomolov properties}
\alttitle{Les propri\'et\'es dynamiques g\'eom\'etriques de Northcott et Bogomolov}
\author{Thomas Gauthier}
\address{CMLS, Ecole Polytechnique, Institut Polytechnique de Paris, 91128 Palaiseau Cedex, France}
\email{thomas.gauthier@polytechnique.edu}
\author{Gabriel Vigny}
\address{LAMFA, Universit\'e de Picardie Jules Verne, 33 rue Saint-Leu, 80039 AMIENS Cedex 1, FRANCE}
\email{gabriel.vigny@u-picardie.fr}
\thanks{Both authors are partially supported by the ANR grant Fatou ANR-17-CE40-0002-01.}
\thanks{Mathematics~Subject~Classification~(2010):  	37P15, 37P30, 11G50, 37P35, 37F45}

\begin{document}
\begin{abstract}
 We establish the dynamical Northcott property for polarized endomorphisms of a projective variety over a function field $\mathbf{K}$ of characteristic zero, and we relate this property to the notion of stability in complex dynamics. This extends previous results of Benedetto, Baker and DeMarco in dimension $1$, and of Chatzidakis-Hrushovski in higher dimension. Our proof uses complex dynamics arguments and does not rely on the previous ones. 
 
 We first show that, when $\mathbf{K}$ is the field of rational functions of a normal complex projective variety, the canonical height of a subvariety is the mass of an appropriate bifurcation current and that a marked point is stable if and only if its canonical height is zero. We then establish the geometric dynamical Northcott property characterizing points of height zero in this setting, using a similarity argument. Moving from points to subvarieties, we propose, for polarized endomorphisms, a dynamical version of the geometric Bogomolov conjecture, recently proved by Cantat, Gao, Habegger, and Xie in the original setting of abelian varieties.
 \end{abstract}
 
 \begin{altabstract}
Nous \'etablissons la propri\'et\'e dynamique de Northcott pour les endomorphismes polaris\'es d'une vari\'et\'e projective sur un corps de fonctions $\mathbf{K}$ de caract\'eristique nulle, et nous relions cette propri\'et\'e \`a la notion de stabilit\'e en dynamique complexe. Cela généralise des r\'esultats de Benedetto, Baker, et DeMarco en dimension $1$, et de Chatzidakis-Hrushovski en dimension plus grande. Notre d\'emonstration met en jeu des arguments de dynamique complexe et ne repose pas sur ceux utilis\'es jusqu'alors.
 
Dans un premier temps nous prouvons que, lorsque $\mathbf{K}$ est le corps des fonctions rationnelles d'une vari\'et\'e projective complexe, la hauteur canonique d'une sous-vari\'et\'e co\"incide avec la masse d'un courant de bifurcation appropri\'e, et qu'un point marqu\'e est stable si et seulement si sa hauteur canonique est nulle. Nous \'etablissons alors la propri\'et\'e dynamique g\'eom\'etrique de Northcott caract\'erisant les points de hauteur nulle dans ce contexte, en utilisant un argument de similarit\'e. En passant de l'\'etude des points \`a celle des sous-vari\'et\'es, nous proposons, pour les endomorphismes polaris\'es, une version dynamique de la conjecture de Bogomolov g\'eom\'etrique, r\'ecemment d\'emontr\'ee par Cantat, Gao, Habegger, et Xie dans le cas des vari\'et\'es ab\'eliennes.
 \end{altabstract}

\maketitle

\section{Introduction}
\subsection{Polarized endomorphisms, functions fields}

Fix a field $\mathbf{K}$ of characteristic zero. A \emph{polarized endomorphism} over $\mathbf{K}$ is a triple $(X,f,L)$ where
	\begin{enumerate}
		\item $X$ is a projective variety defined over $\mathbf{K}$, irreducible over $\bar{\mathbf{K}}$.
		\item $L$ is an ample line bundle of $X$ which is defined over $\mathbf{K}$, and
		\item $f:X\to X$ is an endomorphism defined over $\mathbf{K}$ which is \emph{polarized} by $L$, i.e.\ there is an integer $d\geq2$ such that $f^*L$ is linearly equivalent to $L^{\otimes d}$, denoted $f^*L\simeq L^{\otimes d}$.
	\end{enumerate}
	The integer $d$ is the \emph{degree} of $(X,f,L)$.  A prototypical example is an endomorphism $f$ of a projective space of degree $d\geq2$, since it satisfies $f^*\mathcal{O}(1)\simeq \mathcal{O}(d)\simeq\mathcal{O}(1)^{\otimes d}$.

\medskip

In this article, we are interested in the case where $\mathbf{k}$ is an algebraically closed field of characteristic zero and $\mathbf{K}:=\mathbf{k}(\mathcal{B})$ is the field of rational functions of an irreducible normal projective $\mathbf{k}$-variety $\mathcal{B}$ of dimension at least one (for our purpose, we can always work with the algebraic closure of $\mathbf{k}$ so we directly assume it is algebraically closed). To such a variety $X$ endowed with the ample line bundle $L$, we can associate a \emph{model} $(\mathcal{X},\mathcal{L})$ of $(X,L)$, i.e.\ a surjective morphism 
\[\pi:\mathcal{X}\longrightarrow \mathcal{B}\]
between projective varieties, where $\mathcal{L}$ is a relatively ample line bundle, such that
\begin{enumerate}
	\item the generic fiber of $\mathcal{X}$ is isomorphic to $X$,
	\item the line bundle $L$ is isomorphic to the restriction of $\mathcal{L}$ to the generic fiber,
	\item there exists a dense Zariski open set $\Lambda\subset\mathcal{B}$ over which $\pi$ is flat. We denote $\mathcal{X}_\Lambda:= \pi^{-1}(\Lambda)$.
\end{enumerate}
Note that such a model is not unique, there are infinitely many choices (see \cite[Chapter 1]{Neron-model} for generalities on models of $\mathbf{K}$-schemes). Similarly, $\Lambda$ is not unique and we will shrink it if necessary.
For any $\lambda\in \Lambda(\mathbf{k})$, the fiber $X_\lambda:=\pi^{-1}\{\lambda\}$ is a projective $\mathbf{k}$-variety of dimension $\dim X$ and $L_\lambda:=\mathcal{L}|_{X_\lambda}$ is an ample line bundle of $X_\lambda$.
Furthermore, when $X$ is normal, up to replacing $\mathcal{X}$ by its normalization, we can assume $\mathcal{X}$ is normal and $X_\lambda$ is normal for all $\lambda\in \Lambda$.  We call such a normal variety $\mathcal{X}$ a \emph{normal model} of $X$. 

\medskip

Then, a polarized morphism $f$ defined over $\mathbf{K}$ induces a dominant rational map $\mathcal{f}:\mathcal{X}\dasharrow \mathcal{X}$ and we can choose a dense Zariski open subset $\Lambda\subset\mathcal{B}$ as above such that in addition 
\begin{enumerate}
	\item[(a)] the following diagram commutes
	\[
	\xymatrix{
		\mathcal{X}\ar@{-->}[rr]^{\mathcal{f}}\ar[rd]^{\pi} & & \ar[ld]^{\pi}\mathcal{X}\\
		& \mathcal{B} & }
	\]
	\item[(b)] $\mathcal{f}|_{\mathcal{X}_\Lambda}:\mathcal{X}_\Lambda\to\mathcal{X}_\Lambda$ is a morphism,
	\item[(c)] for any $\lambda\in\Lambda$, if we set $f_\lambda:=\mathcal{f}|_{X_\lambda}$, then $(X_\lambda,f_\lambda,L_\lambda)$ is a polarized endomorphism over $\mathbf{k}$, 
	\item[(d)] $\mathcal{f}$ restricted to the generic fiber of $\mathcal{X}$ can be identified with $f$ via the isomorphism between the generic fiber of $\mathcal{X}$ and $X$ given by Property 1 above. 
\end{enumerate}
\begin{definition} Let $(X,f,L)$ be a polarized endomorphism over $\mathbf{K}=\mathbf{k}(\mathcal{B})$ where $\mathbf{k}$ is an algebraically closed field  of characteristic zero and $\mathbf{K}$ is the field of rational functions of a normal projective $\mathbf{k}$-variety $\mathcal{B}$.
	A triple $(\mathcal{X},\mathcal{f},\mathcal{L})$ satisfying properties \emph{1.$-$3.} and \emph{(a)$-$(d)} above for some dense Zariski open set $\Lambda$ is called an \emph{algebraic family} of polarized endomorphisms; and $\Lambda$ is called a regular part.  Such a family is a \emph{model} of $(X,f,L)$.
\end{definition}
We will frequently shrink the open set $\Lambda$, but we will always assume it is dense and Zariski open.
 Note that in the article, we will use curly letters $\mathcal{f}, \mathcal{L}, \mathcal{Z} \dots$ for objects defined on the model $\mathcal{X}$ in order to distinguish them from their counterparts $f, L, Z \dots$  on $X$.  
When $X$ is not normal, letting $\mathsf{n}:\hat{X}\to X$ be its normalization, the universal property of normalization implies that $f$ lifts to an endomorphism $\hat{f}:\hat{X}\to \hat{X}$ and, since $\mathsf{n}$ is finite, $\mathsf{n}^*L$ is ample, so that $(\hat{X},\hat{f}, \mathsf{n}^*L)$ still defines a polarized endomorphism.

A \emph{marked point} is a rational section $a:\mathcal{B}\dasharrow \mathcal{X}$ whose indeterminacy locus is contained in $\mathcal{B}\backslash \Lambda$, where $\Lambda\subseteq\mathcal{B}$ is a regular part for $(\mathcal{X},\mathcal{f},\mathcal{L})$ (when $\mathcal{B}$ is a curve, the indeterminacy locus is empty). In particular, $a$ defines a regular map $a:\Lambda\to\mathcal{X}$ such that $a(\lambda)\in X_\lambda$ for all $\lambda\in\Lambda$.
To any subvariety $Z$ of $X$ which is defined over $\mathbf{K}$, we can associate a subvariety $\mathcal{Z}$ such that $\pi|_{\mathcal{Z}}:\mathcal{Z}\rightarrow\mathcal{B}$ is flat over a dense Zariski open subset of $\Lambda$. A point of $X(\mathbf{K})$ corresponds to  a marked point $a:\mathcal{B}\dasharrow \mathcal{X}$. We also need the following notion of isotriviality which means, in some sense, that the family does not really depend on the parameter.  
\begin{definition}\label{def_isotrivial} Let $(X,f,L)$ be a polarized endomorphism over $\mathbf{K}=\mathbf{k}(\mathcal{B})$ where $\mathbf{k}$ is an algebraically closed field  of characteristic zero and $\mathbf{K}$ is the field of rational functions of a normal projective $\mathbf{k}$-variety $\mathcal{B}$. Let $(\mathcal{X},\mathcal{f},\mathcal{L})$ be a model of $(X,f,L)$ and let $\Lambda$ be a regular part. When $X$ is normal, we say $(X,f,L)$, or equivalently $(\mathcal{X},\mathcal{f},\mathcal{L})$, is \emph{isotrivial} over $\Lambda$ if, for any $\lambda_0\in \Lambda(\mathbf{k})$, up to making a base change, there is an isomorphism $\phi:\mathcal{X}_\Lambda\to X_{\lambda_0}\times \Lambda$ such that $\phi\circ f=(f_{\lambda_0},\pi)\circ \phi$ and such that for any $\lambda\in \Lambda(\mathbf{k})$, if $\phi_\lambda:=\phi|_{X_\lambda}:X_\lambda\to X_{\lambda_0}$ is the induced isomorphism, then $\phi_\lambda^*L_{\lambda_0}\simeq L_\lambda$.
	
	In the general case, we say $(X,f,L)$ is \emph{isotrivial} over $\Lambda$ if  $(\hat{X}, \hat{f}, \mathsf{n}^*L)$ is, where $\mathsf{n}:\hat{X}\to X$ is its normalization and $\hat{f}:\hat{X}\to \hat{X}$ is the lift of $f$.
\end{definition}

\begin{remark}\normalfont
	\begin{enumerate}
		\item If $(X,f,L)$ is isotrivial over $\Lambda$, so is the polarized variety $(X,L)$ (i.e.\ in a model, fibers over $\lambda$ endowed with their polarization are isomorphic) and isotriviality does not depend on the choice of the model. 
		\item  Lemma~\ref{lm:isotriviality-iterate} states that the notion of isotriviality is independent of the chosen regular part $\Lambda$, and also that a polarized endomorphism $(X,f,L)$ is isotrivial if and only if one of its iterates $(X,f^n,L)$ is. Isotriviality is therefore a dynamically relevant property. 
		\item When $\mathcal{X}=\p^k\times\mathcal{B}$, the condition $\phi^*L_{\lambda'}\simeq L_\lambda$ is automatically satisfied. Indeed, we have $L=\mathcal{O}_{\p^k}(j)$, for some $j\in \N^*$, so that for any $\lambda\in \Lambda(\mathbf{k})$, we have $L_{\lambda}\simeq\mathcal{O}_{\p^k}(j)$. In this case, the isotriviality of $(\mathcal{X},\mathcal{f},\mathcal{L})$ implies the existence, for any $\lambda,\lambda'\in\Lambda$, of a linear isomorphism $\phi_\lambda^{\lambda'}:\p^k\to\p^k$ such that $\phi_\lambda^{\lambda'}\circ f_\lambda=f_{\lambda'}\circ\phi_\lambda^{\lambda'}$.   
		\item When $\dim X=0$, $X(\mathbf{K})$ is a finite set and the isotriviality condition is always satisfied over the Zariski open subset $\Lambda'$ over which $\mathcal{X}\to\mathcal{B}$ is unramified.
	\end{enumerate}
\end{remark}

\subsection{Dynamical stability}
Assume furthermore that $\mathbf{K}=\C(\mathcal{B})$ is the field of rational functions of a normal complex projective variety $\mathcal{B}$.
A marked point $a$ is said \emph{stable}, or equivalently the tuple $((\mathcal{X},\mathcal{f},\mathcal{L}),a)$ is said \emph{stable}, if the sequence $\{\lambda\mapsto f_\lambda^n(a(\lambda))\}_n$ is locally equicontinuous on $\Lambda$. By the Bishop Theorem~\cite[\S 15.5, p. 203]{Chirka}, this means that the volume of the graph $\Gamma_n$ of $\lambda\mapsto f^n_\lambda(a(\lambda))$ seen as a subvariety of $\mathcal{X}_\Lambda$ is bounded independently of $n$ above any relatively compact open subset of $\Lambda$ (i.e.\ locally on $\Lambda$). This notion is, by nature, a local notion. In particular, it is independent of the regular part $\Lambda$.

When $\dim\mathcal{B}=1$ and $X=\p^1$, then $\mathcal{f}$ defines a family of endomorphisms of degree $d$ of $\p^1(\C)$, parametrized by a smooth complex quasi-projective curve $\Lambda$, i.e.\   a morphism $\mathcal{f}:(z,\lambda)\in \p^1\times\Lambda\longmapsto(f_\lambda(z),\lambda)\in \p^1\times\Lambda$ such that for any $\lambda\in\Lambda$, $f_\lambda$ is a rational map of degree $d$. Let $\mathrm{Crit}(f_\lambda)$ denote the critical set of $f$. Up to a base change, the critical points are marked points. McMullen \cite{McMullen-families} considered the case where all these marked points are stable and showed the fundamental result:  either $f$ is isotrivial, or $f$ is a family of Latt\`es maps, in which case the critical points are preperiodic. 

Dujardin and Favre~\cite{favredujardin} then considered the case of a single given marked critical point showing that stability implies preperiodicity or isotriviality. Finally, DeMarco~\cite{demarco} proved such a statement for any marked point.

\subsection{Canonical height and the Northcott property}
We now focus on a more arithmetic point of view. 
A field $\mathbf{K}$ is a {\em product formula field} if $\mathbf{K}$ is equipped with the possibly uncountable family $M_\mathbf{K}$ of
all places of $\mathbf{K}$, a family $(|\cdot|_v)_{v\in M_\mathbf{K}}$ of 
non-trivial absolute values $|\cdot|_v$ representing $v$,
and a family $(N_v)_{v\in M_\mathbf{K}}$ of positive integers satisfying the 
{\em product formula property} in that, for every $z\in \mathbf{K}^*$,
\begin{center}
	$|z|_v=1$ for all but finitely many $v\in M_\mathbf{K}$, and
	$\prod_{v\in M_\mathbf{K}}|z|^{N_\nu}_v=1$.
\end{center}
A place $v\in M_\mathbf{K}$ is said to be finite (resp.\ infinite)
if $|\cdot|_v$ is non-archimedean (resp.\ archimedean).
If $M_\mathbf{K}$ contains an infinite place, then $\mathbf{K}$ is a
\emph{number field} and, if not, it is a \emph{function field}. In particular, when $\mathcal{B}$ is a normal complex projective variety, $\mathbf{K}=\mathbb{C}(\mathcal{B})$ is a product formula field and the choice of the family $(N_v)_{v\in M_\mathbf{K}}$ is equivalent to the choice of a polarization on $\mathcal{B}$, see, e.g.,~\cite[p. 21]{bombieri-gubler}.

Fix a polarized endomorphism $(X,f,L)$ defined over a product formula field $\mathbf{K}$. Let $h_{X,L}$ be the standard Weil height function on $X(\bar{\mathbf{K}})$, relative to the line bundle $L$ (and of a polarization on $\mathcal{B}$). Such a function is defined up to bounded functions and, by the functorial properties of Weil height functions, we have $h_{X,L}\circ f=dh_{X,L}+O(1)$ (\cite{bombieri-gubler}). Following an argument of Tate, Call and Silverman \cite{CS-height} defined the \emph{canonical height} $\widehat{h}_f:X(\bar{\mathbf{K}})\to\R_+$ of $f$ as
\[\widehat{h}_f=\lim_{n\to\infty}\frac{1}{d^n}h_{X,L}\circ f^n.\]
When $\mathbf{K}$ is a number field, Northcott's Theorem \cite{northcott} states that, for any $\varepsilon>0$, the set $\{ x\in X(\mathbf{K}),  \ h_{L,X}(x)\leq \varepsilon\}$ is finite. Since $\widehat{h}_f=h_{X,L} + O(1)$, this implies that the set $\{ x\in X(\mathbf{K}),  \ \widehat{h}_f(x)\leq \varepsilon\}$ is also finite. As $\widehat{h}_f(x)\geq 0$ for all $x$, there is an $\varepsilon_0>0$ 
such that  $\{ x\in X(\mathbf{K}),  \ \widehat{h}_f(x)\leq \varepsilon_0\}= \{ x\in X(\mathbf{K}),  \ \widehat{h}_f(x)=0\}$. Now, since $\widehat{h}_f\circ f= d \widehat{h}_f$, the finite set $\{ x\in X(\mathbf{K}),  \ \widehat{h}_f(x)=0\}$ is $f$-invariant so its points are preperiodic; thus, $  \{ x\in X(\mathbf{K}),  \ \widehat{h}_f(x)\leq \varepsilon_0\}$ is a finite, invariant set of preperiodic points.

 However, when $\mathbf{K}$ is the function field of a projective variety over a field of characteristic zero, there are infinitely many elements of standard height $0$ in $\mathbb{P}^k(\mathbf{K})$ (the elements defined over the field of constants) and all of them have uniformly bounded canonical height since $\widehat{h}_f=h_{X,L} + O(1)$. In particular, to try and state a Northcott property, one uses the canonical height instead of the standard height. One also replaces ``$\forall \varepsilon >0$'' by ``$\exists \varepsilon_0 >0$''. Furthermore, Northcott's property fails when there is isotriviality. For example, if $f$ is defined over the field of constants $\C \subset \C(t)$, then any element $z$ of degree $0$ is sent by $f$ to another element of degree $0$ so its canonical height is zero.

Over $\p^1(\mathbf{K})$, we have 
\begin{theorem}[Northcott property \cite{Baker-functionfield}]\label{tm:Baker}
	Let $\mathbf{K}$ be a function field of characteristic zero and assume  $f\in \
	\operatorname{End}(\p^1_\mathbf{K})$ is non-isotrivial and of degree $d\geq 2$. Then, there exists $\varepsilon_0>0$ such that 
	\[\#\{z\in\p^1(\mathbf{K}), \widehat{h}_{f}(z)\leq \varepsilon_0\}<\infty.\] 
	In particular, any $z\in\p^1(\mathbf{K})$ with $\widehat{h}_{f}(z)=0$ is preperiodic under iteration of $f$, i.e.\  there exist $n>m\geq0$ such that $f^n(z)=f^m(z)$. 
\end{theorem}
This result of Baker holds in arbitrary characteristic (for a suitable definition of isotriviality) and was already proved by Benedetto~\cite{benedetto} for polynomials.  
In the case of abelian varieties over a global function field endowed with the multiplication by an integer $\geq 2$, the analogue of Baker's Theorem is due to Lang-N\'eron \cite{Lang-Neron}. In higher dimension, Chatzidakis and Hrushovski proved a model-theoretic version of Theorem~\ref{tm:Baker} in~\cite{Chatzidakis-Hrushovski} with $\varepsilon_0=0$ in arbitrary characteristic. They only consider polarized endomorphisms of smooth varieties which are primitive (see also \cite{ACL-dynamics} for a detailed exposition on this work). 

Furthermore, when $\mathbf{K}=\C(\mathcal{B})$ for $\mathcal{B}$ a smooth complex projective curve and $f\in\operatorname{End}(\p^1_\mathbf{K})$, DeMarco~\cite{demarco} gave a different proof. In addition, DeMarco showed that the (dynamical) stability of a marked point $a$ implies that the corresponding point $\mathrm{a}\in\mathbb{P}^1(\mathbf{K})$, satisfies $\widehat{h}_f(\mathrm{a})=0$. 
In this case, the Call-Silverman equality $\widehat{h}_f= h_{X,L} + O(1)$ shows that $\widehat{h}_f(\mathrm{a})=0$ if and only if the volume of the Zariski closure of the graph of the marked point $f^n(a)$ in $\mathcal{X}$ is bounded (for a given polarization). Having height $0$ is thus a global property, which implies dynamical stability; we shall see it is equivalent to stability.

\subsection{The geometric dynamical Northcott property and dynamical stability}

 Our first purpose is to give a new proof of the Northcott property on function fields of characteristic zero in higher dimension without the primitive algebraic assumption of \cite{Chatzidakis-Hrushovski}. When $\mathbf{K}=\C(\mathcal{B})$, we also extend DeMarco's results \cite{demarco} and show that a marked point is stable if and only if its canonical height is $0$. 
 
In the sequel, if $\mathcal{Y}$ is a subvariety of $\mathcal{X}$ which is flat over a regular part $\Lambda$, we denote by $\mathcal{f}(\mathcal{Y})$ the Zariski closure in $\mathcal{X}$ of $\mathcal{f}(\mathcal{Y}_\Lambda)$  so that the condition $\mathcal{f}(\mathcal{Y})= \mathcal{Y}$ in the theorem below is meaningful.
 \begin{Theorem}\label{tm:Pk_intro}
	Let $(\mathcal{X},\mathcal{f},\mathcal{L})$ be a non-isotrivial complex algebraic family of polarized endomorphisms over a normal complex projective variety $\mathcal{B}$. Let $(X,f,L)$ be the induced polarized endomorphism over the field $\mathbf{K}=\C(\mathcal{B})$. Let $\mathcal{N}$ be a very ample line bundle on $\mathcal{B}$ such that $\mathcal{M}:=\mathcal{L}\otimes \pi^*(\mathcal{N})$ is ample on $\mathcal{X}$. 
	\begin{enumerate}
		\item Let  $a:\mathcal{B}\dashrightarrow\mathcal{X}$ be a marked point with corresponding point $\mathrm{a}\in X(\C(\mathcal{B}))$, and let $\mathcal{C}_n(a)$ be the Zariski closure in $\mathcal{X}$ of the rational section $a_n:\lambda\mapsto f_\lambda^n(a(\lambda))$. Then
		\begin{equation}\label{stable_height}\tag{$\star$}
		a \ \mathrm{is \ stable} \iff \sup_{n\geq0} \deg_\mathcal{M}(\mathcal{C}_n(a))<\infty \iff \widehat{h}_f(\mathrm{a})=0.
		\end{equation}	
		\item  There exist a regular part $\Lambda$, a (possibly empty) proper subvariety $\mathcal{Y}$ of $\mathcal{X}$which is flat over $\Lambda$, with $\mathcal{f}(\mathcal{Y})=\mathcal{Y}$,  and integers $N\geq0$ and $D_0\geq1$ such that
\begin{enumerate}
\item for any irreducible component $\mathcal{V}$ of $\mathcal{Y}$, there exists $n\in \N^*$ with $\mathcal{f}^{n}(\mathcal{V})=\mathcal{V}$, and the family of polarized endomorphisms $(\mathcal{V},\mathcal{f}^{n}|_\mathcal{V},\mathcal{L}|_\mathcal{V})$ is isotrivial over $\Lambda$.
\item For any marked point $a:\mathcal{B}\dashrightarrow\mathcal{X}$ such that $((\mathcal{X},\mathcal{f},\mathcal{L}),a)$ is stable, and  which is regular over $\Lambda$, then:
\begin{enumerate}
\item the section $a_N:\lambda\mapsto f_\lambda^N(a(\lambda))$ is in fact a section of $\pi|_{\mathcal{Y}}:\mathcal{Y}\to\mathcal{B}$,
\item the Zariski closure $\mathcal{C}_0(a)$ of $a(\Lambda)$ satisfies $\deg_\mathcal{M}(\mathcal{C}_0(a))\leq D_0$.
\end{enumerate}
\end{enumerate}
	\end{enumerate}
\end{Theorem}

Using classical arguments, we can deduce the following corollary, which is what we call the \emph{geometric dynamical Northcott property} on arbitrary function fields of characteristic zero. 
\begin{corollary}\label{cor:Northcott_function_field} Let $\mathbf{K}$ be a function field of characteristic zero. Let $(X,f,L)$ be a non isotrivial polarized endomorphism on a normal projective variety defined over $\mathbf{K}$.
 Then, there exist a (possibly empty or reducible) subvariety $Y\subset X_{\mathbf{K}}$ with $f(Y)= Y$ and an integer $N\geq0$ such that
\begin{enumerate}
\item for any irreducible component $V$ of $Y$, there is $n\geq1$ such that $f^{n}(V)=V$, and the polarized endomorphism $(V,f^{n}|_V,L|_V)$ is isotrivial,
\item for any point $z\in X(\mathbf{K})$ with $\widehat{h}_{f}(z)=0$, we have $f^N(z)\in Y$.
\end{enumerate}
In particular, if $X_{\mathbf{K}}$ contains no periodic isotrivial subvariety of positive dimension, then $\{ z\in X(\mathbf{K}), \ \widehat{h}_{f}(z)=0 \}$ is finite and consists only of preperiodic points.
\end{corollary}

\begin{remark}\normalfont
\begin{enumerate}
\item The case where $Y$ has an irreducible component of positive dimension over $\mathbf{K}$ cannot be excluded. Consider, for example, an  endomorphism of $\mathbb{P}^2_{\mathbf{K}}$ given by a non-isotrivial polynomial map of $\mathbb{A}^2_{\mathbf{K}}$, whose restriction to the line at infinity is independent of $\lambda$ (e.g.\ $f_\lambda(z_1,z_2):=(z_1^d,z_2^d+\lambda z_1)$ defined over $\mathbf{K}:=\mathbb{C}(\lambda)$). Then any non-periodic constant marked point with value on the line at infinity has canonical height zero.
\item Compared to Baker's Theorem~\ref{tm:Baker}, we focus on points of height zero. Our strategy is different from Baker's proof and we did not manage to go further. Recently, Y. Zhang proved a height gap in the case when $X=\mathbb{P}^2$ (see~\cite[Corollary~1.3]{yugang-polarized}).
\end{enumerate}
\end{remark}

\subsection{Strategy of the proof and dynamical height in term of the Green current.}
 We use the notation of Theorem~\ref{tm:Pk_intro}.  Let $\omega_\mathcal{B}$ be a K\"ahler form on $\mathcal{B}$ with cohomology class equal to $c_1(\mathcal{N})$ and let $\widehat{\omega}$ be a continuous closed positive $(1,1)$-form on $\mathcal{X}$ cohomologous to $c_1(\mathcal{L})$.

 To $(\mathcal{X},\mathcal{f},\mathcal{L})$, in \S2 we associate its \emph{fibered Green current} $\widehat{T}_\mathcal{f}$ (see \cite{Demailly_SMF} for currents on singular varieties). This classical object in complex dynamics satisfies the next properties:
 \begin{enumerate}
\item $\widehat{T}_\mathcal{f}$  is a closed, positive current on $\mathcal{X}$ of bidegree $(1,1)$,
\item as a $(k+1,k+1)$ current, $\widehat{T}_\mathcal{f}^{\dim X+1}=0$, see,~e.g.~\cite[Remark~4.7]{BB1},
\item for each $\lambda\in\Lambda$, the slice $\mu_{f_\lambda}$ of $\widehat{T}_\mathcal{f}^{\dim X}$ above $\lambda$ is the unique maximal entropy measure of $f_\lambda$, and
\item for each $\lambda\in\Lambda$, the periodic points of $f_\lambda$ equidistribute towards the measure $\mu_{f_\lambda}$. 
 \end{enumerate}

We now extend the notion of height from points to subvarieties. 
Take any irreducible subvariety $Z$ of dimension $0\leq \ell\leq k=\dim X$ of $X$ defined over $\mathbf{K}$ and let $\mathcal{Z}$ be the associated subvariety  of $\mathcal{X}$ so $\dim (\mathcal{Z}) = \ell + \dim(\mathcal{B})$. We denote by $[\mathcal{Z}]$ the current of integration on  $\mathcal{Z}$. Since $\mathcal{Z}$ is flat over a Zariski open subset of $\mathcal{B}$ which can be taken to be a regular part, $\mathcal{f}^n$ is a rational map which is well defined on a Zariski open set of $\mathcal{Z}$. Recall that $\mathcal{f}^n_*([\mathcal{Z}])$ the push-forward of $[\mathcal{Z}]$ by $\mathcal{f}^n$, defined as the adjoint operator of pull-back by $f^n$ on compactly supported smooth forms, is well defined and coincides with a multiple of $[\mathcal{f}^n(\mathcal{Z})]$, taking into account the multiplicity (see paragraph~\ref{multiplicity}).

 As used in \cite{Gubler,Gubler2,Faber,Cantat-Gao-Habegger-Xie}, the notion of height  $h_{X,L}$ can be extended to the irreducible subvariety $Z$ of dimension $\ell$ as the intersection number 
 \[ h_{X,L}(Z):=\left(\mathcal{Z}\cdot c_1(\mathcal{L})^{\ell+1}\cdot c_1(\pi^*\mathcal{N})^{\dim\mathcal{B}-1}\right),\] 
and it can be computed as 
	\begin{align*}
	h_{X,L}(Z) =\int_{\mathcal{X}_\Lambda}[\mathcal{Z}]\wedge\widehat{\omega}^{\ell+1}\wedge (\pi^*\omega_\mathcal{B})^{\dim\mathcal{B}-1}.
	\end{align*}
Note again that the height function $h_{X,L}$ depends not only on $L$ but also on the polarization $\mathcal{N}$ on $\mathcal{B}$. As in \cite[Theorems 10.9 \& 11.14]{Gubler}, the canonical height  $\hat{h}_{f}(Z)$ of $Z$  is defined by
\[ \hat{h}_{f}(Z) =\lim_{n\to \infty} \frac{1}{d^{(\ell+1)n }} h_{X,L}((f_*)^n(Z)).\]
 So, one could hope that 
\begin{align*}
\hat{h}_{f}(Z)  &= \lim_{n\to+\infty}  \frac{1}{d^{(\ell+1)n }} \int_{\mathcal{X}_\Lambda} (\mathcal{f}^n)_*[\mathcal{Z}]\wedge\widehat{\omega}^{\ell+1}\wedge (\pi^*\omega_\mathcal{B})^{\dim\mathcal{B}-1} \\
 &= \int_{\mathcal{X}_\Lambda} \lim_{n\to+\infty}  \frac{1}{d^{(\ell+1)n }} [\mathcal{Z}]\wedge(\mathcal{f}^n)^*(\widehat{\omega}^{\ell+1})\wedge (\pi^*\omega_\mathcal{B})^{\dim\mathcal{B}-1}, 
\end{align*}
and use the convergence $\lim_{n\to \infty} d^{-(\ell+1)n } (\mathcal{f}^n)^*(\widehat{\omega}^{\ell+1})\to \widehat{T}_\mathcal{f}^{\ell+1}$ in the sense of currents in order to get the following theorem.

\begin{Theorem}\label{tm:formulaheight} 
Let $\mathbf{K}$ be the field of rational functions of a normal complex projective variety $\mathcal{B}$. Let $(X,f,L)$  be a polarized endomorphism over $\mathbf{K}$. Let $(\mathcal{X},\mathcal{f},\mathcal{L})$ be an associated model with regular part $\Lambda$. Let $\mathcal{N}$ be an ample line bundle on $\mathcal{B}$ so that $\mathcal{M}:=\mathcal{L}\otimes \pi^*(\mathcal{N})$ is ample. Let $Z$ be any irreducible subvariety of dimension $0\leq \ell\leq k=\dim X$ of $X$ defined over $\bar{\mathbf{K}}$. Let $\mathcal{Z}$ be the corresponding subvariety of $\mathcal{X}$. Then, the canonical height satisfies
\[\widehat{h}_f(Z)=\int_{\mathcal{X}_\Lambda}[\mathcal{Z}]\wedge\widehat{T}_\mathcal{f}^{\ell+1}\wedge(\pi^*\omega_\mathcal{B})^{\dim\mathcal{B}-1},\]
where $\omega_{\mathcal{B}}$ is any K\"ahler form on $\mathcal{B}$ cohomologous to $c_1(\mathcal{N})$.
\end{Theorem}
To make the above strategy work, one has to deal with the difficulty that some mass might be added on the boundary $\pi^{-1}(\mathcal{B}\backslash \Lambda)$. This possible loss of mass corresponds a priori to the so-called places of bad reduction of the map $f$. DeMarco's idea for $X=\p^1$ is to do analysis at those places in a delicate way \cite[Sections 3 $\&$ 4]{demarco}. In here, we look instead at what happens away from  $\pi^{-1}(\mathcal{B}\backslash \Lambda)$ using an appropriate cut-off function built with a naive degeneration's estimate of potentials of the current $\widehat{T}_\mathcal{f}$ coming from the Nullstellensatz (such an estimate is present in DeMarco's article). This gives a new proof of the existence of the canonical height, originally due to Gubler \cite{Gubler}, and the formula in Theorem~\ref{tm:formulaheight} computes that height for any irreducible subvariety $Z$ of $X$ in term of the fibered Green current. By definition a current is in the dual space of smooth forms with compact support. Hence stability, the vanishing of a current, is local by nature. Theorem~\ref{tm:formulaheight} shows that stability 
 is in fact a global notion. The cut-off function we use is a DSH function, the strength of such objects, introduced by Dinh and Sibony~\cite{DinhSibonysuper}, is that their definition involves the complex structure of the space (which is not the case for $C^\alpha$-functions). 

 The formula for the canonical height in Theorem~\ref{tm:formulaheight} is known in the case of rational sections of elliptic surfaces where $f$ is the multiplication by $n\geq 2$ \cite[Theorem 3.2]{CDMZ} and has been investigated in other special
cases where $X$ is an abelian variety (the Betti form of a family of abelian varieties is exactly the fibered Green current of the fiberwise multiplication by $2$); see for example \cite{Cantat-Gao-Habegger-Xie}.

\medskip

Extending the notion of stability to subvarieties (see \S\ref{sec:bifcur}, Definition~\ref{def_stable} and also Definition~\ref{algebraic_pair}), we shall obtain the following corollary which proves \eqref{stable_height} in Theorem~\ref{tm:Pk_intro} (see \S\ref{multiplicity} for the definition of $\deg(f^n|_Z)$).
\begin{corollary}\label{cor:stable-h0_intro} Under the hypothesis of Theorem~\ref{tm:formulaheight}, the pair $((\mathcal{X},\mathcal{f},\mathcal{L}),[\mathcal{Z}])$ is stable if and only if $\widehat{h}_{f}(Z)=0$ if and only if there exists $C>0$ such that $d^{\ell n}/C\leq \deg_{\mathcal{M}}(\mathcal{f}^n(\mathcal{Z}))\cdot \mathrm{deg}(f^n|_{Z}) \leq C d^{\ell n}$.
\end{corollary}

Corollary~\ref{cor:stable-h0_intro} is a direct consequence of Theorem~\ref{tm:caracterization-stable} below.

\subsection{The geometric dynamical Bogomolov property}

We say that an irreducible subvariety $Z\subset X_{\bar{\mathbf{K}}}$ comes from an \emph{isotrivial factor} of $(X,f,L)$ if there exist 
\begin{itemize}
\item an isotrivial polarized endomorphism $(Y,g,E)$,
\item a subvariety $V\subset X_{\bar{\mathbf{K}}}$ and an integer $k\geq1$ with $f^k(V)=V$,
\item a dominant rational map $p:V\dashrightarrow Y$ such that $p\circ(f^k|_{V})=g\circ p$, and
\item an integer $N\geq1$ and an isotrivial subvariety $W_0\subset Y$, such that $f^N(Z)=p^{-1}(W_0)$. 
\end{itemize}

We also adapt the following definition of Ghioca and Tucker \cite{Ghioca_Tucker} of \emph{special subvariety} in the number field case to the function field case.
\begin{definition}
	With the above notations, we say that $Z$ is \emph{$f$-special} if there exist an integer $n$, a polarized endomorphism $(X,\Psi,L)$ a subvariety $Y$ with $Z\subseteq Y \subseteq X$ such that
	\begin{itemize}
		\item $f^n(Y)=Y=\Psi(Y)$;
		\item $f^n\circ \Psi = \Psi \circ f^n$ on $Y$;
		\item either $Z$ is preperiodic under $\Psi$ or $Z$ comes from an isotrivial factor of $(X,\Psi,L)$.
	\end{itemize}   
	\end{definition}
We propose the following conjecture.
\begin{conjecture}[geometric dynamical Bogomolov]\label{Conjecture_Bogomolov} Let $\mathbf{K}$ be a function field of characteristic zero. Let $(X,f,L)$ be a polarized endomorphism defined over $\mathbf{K}$, where $X$ is normal.
	 Let $Z\subset X_{\bar{\mathbf{K}}}$ be an irreducible subvariety and assume that, for any $\varepsilon>0$, the set $Z_\varepsilon:=\{x\in Z(\bar{\mathbf{K}})\, : \ \widehat{h}_f(x)<\varepsilon\}$
is Zariski dense in $Z$. Then, $Z$ is $f$-special.
\end{conjecture}

Note that, if $X=A$ is a non-isotrivial abelian variety and if $f$ is the multiplication by an integer $n\geq2$, this reduces to the geometric Bogomolov conjecture proposed by Yamaki~\cite[Conjecture 0.3]{Yamaki}, where the isotrivial factor is the $\mathbf{K}/\mathbf{k}$-trace of $A$, whence it is known to hold by \cite{Cantat-Gao-Habegger-Xie} (see also \cite{Xie_Yuan} for the general case of arbitrary characteristic).
Note also that the case of isotrivial factor occurs, e.g.\  if $f$ preserves globally the fibers of a fibration of which $Z$ is a fiber, we actually have $\widehat{h}_f(Z)=0$, even though it may not be preperiodic under iteration of $f$, see Section~\ref{sec:Bogomolov} for more details.

 As observed by Ghioca, Tucker and Zhang \cite{GTZ}, over a number field, if $E$ is a CM elliptic curve and $\omega_1, \omega_2 \in \mathrm{End}(E)$ are multiplicatively independent with $|\omega_1|=|\omega_2|>1$, then  $(\omega_1,\omega_2)$ defines a polarized endomorphism of $E\times E$ for which the diagonal $\Delta$ is not preperiodic but contains infinitely many points of height zero (hence its height is zero). This explains the definition of special subvariety in this setting where one can take $\Psi$ to be $([2],[2])$ and $Y=E\times E$. As explained to us by Fakhruddin, one can construct similar examples over function fields using arguments of Shimura \cite{Shimura}: given an imaginary quadratic field $\mathbb{K}$ and an integer $k> 2$, there exist a non-isotrivial family $\mathcal{A} \to S$ of abelian varieties of relative dimension $k$ with $\dim S >0$ and such that the endomorphism ring of $A_\eta$ is the ring of integers of $\mathbb{K}$.

\medskip

Let $\mathcal{f}:\p^k\times\Lambda\to\p^k\times\Lambda$ be an algebraic family of complex rational maps and let $f:\p^k_\mathbf{K}\to\p^k_\mathbf{K}$ be the induced endomorphism defined over $\mathbf{K}=\C(\Lambda)$. The fundamental work of Berteloot, Bianchi and Dupont \cite{BBD} sets up a good notion of stability, \emph{$J$-stability}: one way to formulate the $J$-stability of the family $\mathcal{f}$ is to require that $\widehat{h}_f(\mathrm{Crit}(f))=0$, where $\mathrm{Crit}(f)$ is the critical divisor of $f$. 
Can one prove a rigidity property for non isotrivial $J$-stable families?
Apart from the dimension $1$ where McMullen did characterize stable algebraic families of rational maps \cite{McMullen-families}, this is largely open.
 Solving Conjecture~\ref{Conjecture_Bogomolov} would be a first step toward proving that such a family is post-critically finite.

\subsection{Notations and conventions}\label{multiplicity}
In the whole article, if $\mathbf{k}$ is any field of characteristic zero, $X$ and $Y$ are $\mathbf{k}$-varieties, $Y$ irreducible, and $f:X\to Y$ is a finite morphism, we let $\deg(f)$ be the topological degree of $f$, i.e. 
\[\deg(f)=\mathrm{Card}\left(\{x\in X(\bar{\mathbf{k}})\, : \ f(x)=y\}\right),\]
for a general point $y\in Y(\bar{\mathbf{k}})$.
For every integral subvariety $Z\subseteq X$, since $f$ is proper, the image of $Z$ is again a subvariety we denote by $f(Z)$. We also let $f_*(Z)$ be the push-forward of $Z$ by $f$. It is the cycle of $Y$ given by
\[f_*(Z)=\deg(f|_Z)\cdot f(Z).\]
Here $\deg(f|_Z)$ is the topological degree of the map $f|_Z:Z\to f(Z)$.

When $\mathbf{k}=\mathbb{C}$ is the field of complex numbers, we also denote by $[Z]$ the integration current on the subvariety $Z$, i.e. for any test form $\psi$ on $X$ of bidegree $(\dim Z,\dim Z)$, we let
\[\int_X [Z]\wedge \psi:=\int_{Z_\mathrm{reg}}\psi|_Z,\]
since a theorem of Lelong says it defines a positive closed current \cite[Theorem 2.7, p.140]{Demailly}.
We also denote by $f_*[Z]$ the push-forward current, i.e. the current defined by
\[\int_{Y} \left(f_*[Z]\right)\wedge \phi=\int_X [Z]\wedge(f^*\phi)\]
for any test form $\phi$ on $Y$, where $f^*\phi$ is the pullback of $\phi$. We can also compute that $f_*[Z]=\deg(f|_Z)\cdot [f(Z)]$. That way, $f_*[Z]= [f_*(Z)]=\deg(f|_Z)[f(Z)]$.

\section{Analytic families of polarized endomorphisms}\label{section2}

Since isotrivial subvarieties can be singular, it is natural to study polarized endomorphisms of possibly singular projective varieties. Nevertheless, we will rely on an important result of Fakhruddin~\cite{Fakhruddin} which says that such polarized morphisms can always be seen as the restriction of a polarized endomorphism defined on some $\p^N$. This will help us build the fibered Green current of a family of polarized endomorphisms. Indeed, it will imply  that we are in the settings of \cite{Demailly_SMF}: we will be dealing with quasi-plurisubharmonic (qpsh for short) functions on $X$ which are the restrictions of qpsh functions defined on a global neighborhood of $X$ in $\p^N(\C)$.

\subsection{On a key result of Fakhruddin}\label{sec:fak}

Let $X$ be an irreducible complex projective variety.
Recall that from a very ample line bundle $L$ on $X$, we can produce an embedding $\iota:X\hookrightarrow \p(H^0(X,L)^\vee)=\p^N$, where $N+1=\dim H^0(X,L)^\vee$ and $L$ may then be identified with the pullback $\iota^*\mathcal{O}(1)$. We also define a hermitian metric on $X$ by setting $\omega_X$ to be the restriction of  $\omega_{\mathrm{FS}}$ to $X$, where $\omega_{\mathrm{FS}}$ is the Fubini-Study form on $\mathbb{P}^N$. In this text, when fixing a K\"ahler form on $X$, we shall always take such a Fubini-Study form $\omega_X$.

\medskip

Thanks to the next result of Fakhruddin~\cite[Corollary~2.2]{Fakhruddin}, we can find an embedding adapted to the study of the dynamics of a polarized endomorphism:

\begin{proposition}\label{prop:fak}
	Let $(X,f,L)$ be a polarized endomorphism over an infinite field $K$. Then there exist $N\geq1$, an embedding $\iota:X\hookrightarrow \p^N_K$, an integer $e\geq1$ and an endomorphism $F:\p^N_K\to\p^N_K$ such that
	\[\iota \circ f=F\circ \iota \ \text{on} \ X \quad \text{and} \quad L^{\otimes e}=\iota^*\mathcal{O}_{\mathbb{P}^N}(1).\]
\end{proposition}
\begin{proof}
Let us explain briefly the construction of $F$, since we rely on it in the sequel. Up to replacing $L$ with $L^{\otimes e}$ we may assume $L$ is very ample. In this case, $L$ induces an embedding $\iota: X\hookrightarrow \mathbb{P}(H^0(X,L)^\vee)=\p^N$ and, one can choose a basis $(s_0,\ldots,s_N)$ of $H^0(X,L)^\vee$ with no common zeros, so that $f_j:=f^*(s_j)$ is a degree $d$ polynomial in the $s_i$'s. We may consider $(s_0,\ldots,s_N)$ as affine coordinates on $\mathbb{A}^{N+1}$, so that the pullback map $f^*$ induces an endomorphism
 $F:\mathbb{P}^N\to \mathbb{P}^N$, which may be defined by $F([s_0:\cdots:s_N])=[f_0:\cdots:f_N]$. 
\end{proof}

\begin{remark}\label{rem_replace_very_ample}
	\normalfont Note that if  $(X,f,L)$ is a polarized endomorphism of degree $d$ and $e \geq 1$, then   $(X,f,L^{\otimes{e}})$ is again  a polarized endomorphism of degree $d$. In particular, we shall often take $L$ to be very ample (which is equivalent to taking $e$ large enough) and we will explain why it does not affect the statements whenever it is needed.  
\end{remark}

In the following, we will say that an embedding $\iota:X\hookrightarrow \mathbb{P}^N_K$ is \emph{adapted} to the pair $(X,f,L)$ if it satisfies the conclusion of Proposition~\ref{prop:fak}. We will also use that one can assume the form $\omega_X$ comes from an adapted embedding.

\begin{corollary}\label{cor:fak}
Let $(X,f,L)$ be a polarized endomorphism over $\mathbb{C}$ of degree $d$. There exists a continuous function $u:X\to\mathbb{R}$ such that 
\[f^*\omega_X=d\cdot\omega_X+\mathrm{dd}^cu.\]
\end{corollary}

\begin{proof}
Let $\iota:X\hookrightarrow \mathbb{P}^N(\mathbb{C})$ and $F:\mathbb{P}^N(\mathbb{C})\to\mathbb{P}^N(\mathbb{C})$ be given by Proposition~\ref{prop:fak}. Then there is a smooth function $\phi:\mathbb{P}^N(\mathbb{C})\to\mathbb{R}$ such that
\[F^*\omega_{\mathrm{FS}}=d\cdot \omega_{\mathrm{FS}}+\mathrm{dd}^c\phi.\]
Since $\omega_X=\iota^*\omega_{\mathrm{FS}}$, we just have to define $u$ as $u:=\phi\circ\iota$.
\end{proof}

\subsection{Dynamics of polarized endomorphisms of projective varieties}\label{dynamics_polarized}
In the rest of the article, a $(p,p)$-current means a current of bidegree $(p,p)$.
Building on a long history of results in complex dynamics, especially ideas of Briend and Duval \cite{briendduval} and their adaptation to the meromorphic setting by Dinh, Nguyen, and Truong \cite{DNT}, we prove the following:

\begin{theorem}\label{prop:endopolarized}
Let $(X,f,L)$ be a polarized endomorphism over $\mathbb{C}$ of degree $d$. Let $k:=\dim(X)$. Then the sequence $(d^{-n}(f^n)^*\omega_X)$ converges in the sense of currents to a closed positive $(1,1)$-current $T_f$. Moreover,
\begin{enumerate}
\item $T_f=\omega_X+\mathrm{dd}^cg_f$ where $g_f$ is continuous on $X$ and $f^*T_f=d\cdot T_f$,
\item the repelling periodic points of $f$ are Zariski dense in $X$.
\end{enumerate}
\end{theorem}

First, we recall the definition of the dynamical degrees following \cite{RS, DinhSibonydegree} (see also \cite{Truong-dyn-degree, Bac-dyn-degree} for arbitrary characteristic). Let $X$ be an irreducible projective variety of dimension $k$, $L$ a big and nef line bundle on $X$ and $f:X\dashrightarrow X$ a dominant rational mapping. 
For any $1\leq j\leq k$, we set 
\[\deg_{j,L}(f)=\left((f^*c_1(L)^j)\cdot c_1(L)^{k-j}\right)/(c_1(L)^k).\]
Then, the $j$-th dynamical degree of $f$ is the quantity
\[\lambda_j(f):=\lim_{n\to\infty}\left(\deg_{j,L}(f^n)\right)^{1/n}.\]
This limit is well-defined, the quantities $\lambda_j(f)$ are birational invariants and do not depend on the choice of the big and nef line bundle (see~\cite{DinhSibonydegree, Truong-dyn-degree, Bac-dyn-degree}).
In particular, $\lambda_k(f)$ is the topological degree of $f$; we say that $f$ has \emph{large topological degree} if $\lambda_j(f)<\lambda_k(f)$ for all $1\leq j<k$.
 When $(X,f,L)$ is polarized of degree $d$, then for any $1\leq j\leq k$ and any $n\geq1$, we have
\[\deg_{j,L}(f^n)=\left(((f^n)^*c_1(L))^j\cdot c_1(L)^{k-j}\right)/(c_1(L)^k)=\left((d^n c_1(L))^j\cdot c_1(L)^{k-j}\right)/(c_1(L)^k)=d^{jn},\]
since $f^n$ is also polarized of degree $d^n$. In particular, $\lambda_j(f)=d^j$ for all $j$ and $f$ has large topological degree $d^k$.

\begin{proof}[Proof of Theorem~\ref{prop:endopolarized}]
Let $u:X\to\mathbb{R}$ be the continuous function given by Corollary~\ref{cor:fak}. By an induction, we find
  \[\frac{1}{d^n}(f^n)^*\omega_X=\omega_X+\mathrm{dd}^c\left(\sum_{j=0}^{n-1}\frac{u\circ f^j}{d^{j}}\right).\]
The sequence $u_n:=\sum_{j=0}^{n-1}\frac{u\circ f^j}{d^{j}}$ of continuous functions converges uniformly on $X$ to a continuous function $g_f$. By construction,  $T_f:=\omega_X+\mathrm{dd}^cg_f$ is the limit of $d^{-n}(f^n)^*\omega_X$ and $f^*T_f=d\cdot T_f$.

To prove point $2$, we let $p:\tilde{X}\to X$ be a resolution of the  singularities of $X$. Then $\tilde{X}$ is a smooth projective variety of dimension $k$ and the morphism $f\circ p:\tilde{X}\to X$ lifts as a rational mapping $\tilde{f}:\tilde{X}\dashrightarrow\tilde{X}$.
Let $\tilde{L}:=p^*(L)$. As $p$ is generically finite and $L$ is ample, $\tilde{L}$ is big and nef, and $\tilde{f}^*\tilde{L}\simeq\tilde{L}^{\otimes d}$ by construction. The mapping $\tilde{f}$ is dominant with large topological degree, since its dynamical degrees are the same as the dynamical degrees of $f$. In particular, by \cite[Theorem 1.1]{DNT} (see also \cite{Gu1, briendduval}), repelling periodic points of $\tilde{f}$ equidistribute towards a probability measure which does not give mass to pluripolar sets. Whence, those contained in $X_\mathrm{reg}$ are Zariski dense in $X$.
\end{proof}

We will also use the next lemma in the sequel, which is proved in \cite[Lemma~2.1]{Nakayama-Zhang}.

\begin{lemma}\label{lm:finite}
Let $(X,f,L)$ be a polarized endomorphism over $\mathbb{C}$ of degree $d$. 
Let $V$ be an invariant subvariety of dimension $\ell$, i.e.\ $f(V)=V$ then $(V, f|_{V}, L|_{V})$ is a polarized endomorphism of degree $d$. 
In particular, $f|_{V}$ is a finite endomorphism of topological degree $d^\ell$, hence $D=d^\ell$.

Consequently, for any $n\geq1$, the set $\{z\in X\, : \ f^n(z)=z\}$ is finite.
\end{lemma}

\begin{proof}
Take $V$ satisfying the hypothesis of the lemma. As $L$ is ample, so is $L|_{V}$, and $(V, f|_{V}, L|_{V})$ is a polarized endomorphism, since by functoriality of pull-back $(f|_V)^* L|_V$ is isomorphic to $(f^*L)|_V$ which is isomorphic to $L^{\otimes d}|_V$. The topological degree $D$ of $f|_{V}$ satisfies 
\[   \left(f|_{V}^*c_1( L|_{V})\right)^{\ell}= D \cdot c_1( L|_{V})^{\ell};\]
since $f|_{V}^*(c_1(L|_{V}))= d c_1(L|_{V})$ we deduce that  $f|_{V}$ is of topological degree $d^\ell$. And $f_{|V}$ is finite by the arguments of Fujimoto \cite[Lemma 2.3]{Fujimoto} (the proof is given in the smooth case, but adapts to the normal case).

Fix $n\geq1$ and assume the subvariety $Y:=\{z\in X\, : \ f^n(z)=z\}$ of $X$ has an irreducible component $W$ of dimension $\ell\geq 1$. As $f^n_{|W}$ is the identity, it has degree $1$ which contradicts the above.
\end{proof}

\subsection{The fibered Green current of a family of endomorphisms}\label{sec:fibered}

We extend here the definitions of \cite{favredujardin} to the case of families of endomorphisms of complex projective varieties. We are interested in local parametric aspects of such families, so we do not rule out transcendental dependence on the parameter, e.g.\ this section applies to the family $f_\lambda(z)=z^2+\exp(\lambda)$, $\lambda\in \mathbb{C}$.

\medskip

Fix a reduced complex analytic space $\mathcal{X}_\Lambda$, a complex Kähler manifold $\Lambda$ and assume there is a proper surjective morphism $\pi:\mathcal{X}_\Lambda\to \Lambda$ which is  an analytic submersion. Assume furthermore, there is a relatively ample line bundle $\mathcal{L}$ on $\mathcal{X}_\Lambda$, i.e.\ $\mathcal{L}$ is a line bundle on $\mathcal{X}_\Lambda$ such that $L_\lambda:=\mathcal{L}|_{X_\lambda}$ is an ample divisor on $X_\lambda:=\pi^{-1}\{\lambda\}$ for all $\lambda$ (see, e.g., \cite[\S 1.7]{Lazarsfeld-positivity}). In particular,  $X_\lambda$ is a complex projective variety of dimension $k$ for all $\lambda\in\Lambda$ which we assume to be irreducible.

 Such a morphism $\pi$ is a \emph{family of irreducible complex projective varieties of dimension $k\geq1$}.

\medskip

\begin{definition}\label{embedding_analytic}
A \emph{family of endomorphisms} of $\mathbb{P}^N$ of degree $d>1$ is a holomorphic map $F:(z,\lambda)\in\mathbb{P}^N\times\Lambda\mapsto(F_\lambda(z),\lambda)\in\mathbb{P}^N\times\Lambda$ such that $F_\lambda$ is an endomorphism of degree $d$ for any $\lambda\in \Lambda$.

A triple $(\mathcal{X}_\Lambda,\mathcal{f},\mathcal{L})$ is a complex (or holomorphic) \emph{family of polarized endomorphisms}, if $\mathcal{f}:\mathcal{X}_\Lambda\to\mathcal{X}_\Lambda$ is analytic and there is an embedding
$\widehat{\iota}:\mathcal{X}_\Lambda\hookrightarrow \mathbb{P}^N\times\Lambda$ and a family $F:\mathbb{P}^N\times\Lambda\to\mathbb{P}^N\times\Lambda$ of endomorphisms such that
\begin{center}
$F\circ \widehat{\iota}=\widehat{\iota}\circ f$ and $\mathcal{L}^{\otimes e}=\widehat{\iota}^*\mathcal{O}_{\mathbb{P}^N}(1)$ for some $e\geq 1$,
\end{center}
and such that $p_2\circ\widehat{\iota}=\pi$, where $p_2:\mathbb{P}^N\times\Lambda\to\Lambda$ is the projection onto the second factor. Then, each  $f_\lambda:=\mathcal{f}|_{X_\lambda}:X_\lambda\to X_\lambda$ is an endomorphism of $X_\lambda$ polarized by $\mathcal{L}|_{X_\lambda}$.
\end{definition}
We let $\omega$ be a K\"ahler form on $\mathbb{P}^N$, $\tilde \omega $ be the induced $(1,1)$-form on $\mathbb{P}^N\times\Lambda$ and we let $\widehat{\omega}:=\widehat{\iota}^*(\tilde\omega)$. Note that $\widehat{\omega}$ is a $(1,1)$-form on $\mathcal{X}_\Lambda$ which is cohomologous to a multiple of $c_1(\mathcal{L})$ and that $\omega_\lambda:=\widehat{\omega}|_{X_\lambda}$ is a K\"ahler form on $X_\lambda$ which comes from an adapted embedding.
\begin{proposition}
Let $(\mathcal{X}_\Lambda,\mathcal{f},\mathcal{L})$ be a complex family of polarized endomorphisms. Then there exists a function $g$, which is the restriction to $\mathcal{X}_\Lambda$ of a continuous qpsh function on $\p^N \times \Lambda$ such that the sequence $d^{-n}(\mathcal{f}^n)^*(\widehat{\omega})$ converges in the sense of currents towards a closed positive $(1,1)$-current $\widehat{T}_\mathcal{f}= \widehat{\omega}+\mathrm{dd}^c g$. Moreover, 
\begin{enumerate}
 \item $\mathcal{f}^*\widehat{T}_\mathcal{f}=d\cdot\widehat{T}_\mathcal{f}$,
 \item for all $\lambda\in\Lambda$, $T_\lambda:=\omega_\lambda+\mathrm{dd}^c(g|_{X_\lambda})$ is a well-defined positive closed $(1,1)$ current on $X_\lambda$, $(f_\lambda)^*T_\lambda=d\cdot T_\lambda$ and $T_\lambda= T_{f_\lambda}$.
 \end{enumerate} 
 \end{proposition} 
 
 \begin{proof} Let $\phi$ be a smooth quasi-potential of $F^*\tilde{\omega}$ (i.e. $F^*\tilde{\omega}=d\cdot \tilde{\omega}+\mathrm{dd}^c\phi$). Since $F\circ \widehat{\iota}=\widehat{\iota}\circ \mathcal{f}$, we have 
\[\mathcal{f}^*\widehat{\omega}=\widehat{\iota}^*(F^*\tilde{\omega})=\widehat{\iota}^*(d\cdot \tilde{\omega}+\mathrm{dd}^c\phi)=d\cdot\widehat{\omega}+\mathrm{dd}^c(\phi\circ\widehat{\iota}).\]
The function $u:=\phi\circ\widehat{\iota}:\mathcal{X}_\Lambda\to\mathbb{R}$ is continuous and satisfies $\mathcal{f}^*\widehat{\omega}=d\cdot \widehat{\omega}+\mathrm{dd}^cu$.
An easy induction then gives
 \[\frac{1}{d^n}(\mathcal{f}^n)^*\widehat{\omega}=\widehat{\omega}+\mathrm{dd}^c\left(\sum_{j=0}^{n-1}\frac{u\circ \mathcal{f}^j}{d^{j}}\right),\]
 and $d^{-n}(\mathcal{f}^n)^*\widehat{\omega}$ converges towards the closed positive $(1,1)$-current 
 \[\widehat{T}_\mathcal{f}:=\widehat{\omega}+\mathrm{dd}^c g\]
 with $g:=\sum_jd^{-j}u\circ \mathcal{f}^j$ and the convergence is locally uniform since $u$ is locally bounded.
 By construction, the function $g$ is the restriction to $\mathcal{X}_\Lambda$ of a continuous qpsh function on $\p^N \times \Lambda$. We have $\mathcal{f}^*\widehat{T}_\mathcal{f}=d\cdot \widehat{T}_\mathcal{f}$, and $\widehat{T}_\mathcal{f}|_{X_\lambda}:=\omega_\lambda+\mathrm{dd}^c(g|_{X_\lambda}) = T_{f_\lambda}$ since  $\omega_\lambda$ is a K\"ahler form on $X_\lambda$ which comes from an adapted embedding.  We thus have proved the proposition.
 \end{proof}
 
 \begin{definition}
The current $\widehat{T}_\mathcal{f}$ is the \emph{fibered Green current} of the family $(\mathcal{X}_\Lambda,\mathcal{f},\mathcal{L})$.
\end{definition}

\subsection{The bifurcation current of a closed positive current}\label{sec:bifcur}

Using the notations of the previous subsection, we now define a notion of stability and a bifurcation current for a closed positive current on $\mathcal{X}_\Lambda$.
Recall that $k$ is the relative dimension of $\mathcal{X}_\Lambda\to \Lambda$ and let $0\leq p \leq k$ be an integer.

\begin{definition}\label{def:dynamical_pair}
A $p$-\emph{measurable dynamical pair} $((\mathcal{X}_\Lambda,\mathcal{f},\mathcal{L}),S)$ on $\mathcal{X}_\Lambda$ parametrized by a complex manifold $\Lambda$ is the datum of a holomorphic family $(\mathcal{X}_\Lambda,\mathcal{f},\mathcal{L})$ as above and a positive closed $(p,p)$-current $S$ on $\mathcal{X}_\Lambda$.
\end{definition}

The \emph{bifurcation current} of a $p$-measurable dynamical pair $((\mathcal{X}_\Lambda,\mathcal{f},\mathcal{L}),S)$ is the positive closed $(1,1)$-current on $\Lambda$ defined as
\begin{equation}\label{def_bifurcation_current_pair}
T_{\mathcal{f},S}:=\pi_*\left(\widehat{T}_\mathcal{f}^{k+1-p}\wedge S\right).\end{equation}
Since $\widehat{T}_\mathcal{f}$ has locally bounded potentials, the above product is well-defined (see~\cite[Chapter III]{Demailly_SMF}). The idea to study bifurcations using currents goes back to the work of DeMarco~\cite{Demarco1} and has been intensively used since then, see e.g.\  \cite{bsurvey} and references therein.

We also endow $\Lambda$ with a K\"ahler form $\omega_\Lambda$ (with a possibly  infinite volume), and let $\check{\omega}_\Lambda:=(\pi)^{*}\omega_\Lambda$. In this case, $\check{\omega}_\Lambda+\widehat{\omega}$ is a K\"ahler form on $(\mathcal{X}_\Lambda)_{\mathrm{reg}}$ and we can define the \emph{mass} of a closed positive $(p,p)$-current $S$ on a compact subset $\mathcal{K}$ of $\mathcal{X}_\Lambda$ as
\[\|S\|_{\mathcal{K}}:=\int_{\mathcal{K}} S\wedge\left(\widehat{\omega}+\check{\omega}_\Lambda\right)^{k+\dim \Lambda-p}<+\infty.\]

\begin{definition}\label{def_stable}
We say that the  $p$-measurable dynamical pair $((\mathcal{X}_\Lambda,\mathcal{f},\mathcal{L}),S)$ is \emph{stable} if the sequence $\{(\mathcal{f}^n)_*(S)\}_{n\geq1}$ of closed positive $(p,p)$-currents satisfies \[\|(\mathcal{f}^{n})_{*}(S)\wedge\check{\omega}_\Lambda^{\dim\Lambda-1}\|_\mathcal{K}=o(d^{n(k+1-p)}) \ \mathrm{as} \ n\to \infty\]
 in any compact subset $\mathcal{K}$ of $\mathcal{X}_\Lambda$.
\end{definition}

\begin{proposition}\label{lm:bifmeasure}
Let $(\mathcal{X}_\Lambda,\mathcal{f},\mathcal{L})$ be any holomorphic family of polarized endomorphisms as above and let $S$ be any positive closed $(p,p)$-current of $\mathcal{X}_\Lambda$. 
Then, the following assertions are equivalent:
\begin{enumerate}
\item the pair $((\mathcal{X}_\Lambda,\mathcal{f},\mathcal{L}),S)$ is stable,
\item for any compact set $\mathcal{K}\subset \mathcal{X}_\Lambda$, we have $\|(\mathcal{f}^{n})_{*}(S)\wedge\check{\omega}_\Lambda^{\dim\Lambda-1}\|_\mathcal{K}=O(d^{n(k-p)})$,
\item as a current on $\Lambda$, we have $T_{\mathcal{f},S}=0$.
\end{enumerate}
\end{proposition}

This proposition is a direct consequence of the next lemma which follows the proof of \cite[Proposition-Definition 3.1]{favredujardin} and of \cite[Lemma 3.13]{BBD}.

\begin{lemma}\label{lm:growthmass}
	Let $S$ be a closed positive $(p,p)$-current on $\mathcal{X}_\Lambda$. For any compact sets $K_1\Subset K_2\Subset\Lambda$, there exists a constant $C(K_1,K_2)>0$ depending only on $K_1$, $K_2$ and on $(\mathcal{X}_\Lambda,\mathcal{f},\mathcal{L})$, such that for any $n\geq0$, the quantity
\[\left|\|({\mathcal{f}}^{n})_*(S)\wedge\check{\omega}_\Lambda^{\dim\Lambda-1}\|_{\pi^{-1}(K_1)}-d^{(k+1-p)n}\int_{\pi^{-1}(K_1)}\widehat{T}_\mathcal{f}^{k+1-p}\wedge S\wedge\check{\omega}_\Lambda^{\dim\Lambda-1}\right|\]
is bounded above by $C(K_1,K_2)\cdot d^{(k-p)n}\|S\wedge\check{\omega}_\Lambda^{\dim\Lambda-1}\|_{\pi^{-1}(K_2)}$.
\end{lemma}

\begin{proof}
Let $K_1\Subset K_2\Subset\Lambda$ be any compact subsets and write
$S_n:=(\mathcal{f}^{n})_*(S)\wedge\check{\omega}_\Lambda^{\dim\Lambda-1}$,
	for any integer $n\geq0$. Since $(\pi\circ \mathcal{f}^{n})^{*}\omega_\Lambda=\pi^*\omega_\Lambda=\check{\omega}_\Lambda$, we have
	\begin{align*}
	\|S_n\|_{\pi^{-1}(K_1)} & = \int_{\pi^{-1}(K_1)} S_0\wedge\left(\check{\omega}_\Lambda+ (\mathcal{f}^{n})^{*}\widehat{\omega}\right)^{k+1-p}.
	\end{align*} 
	Now we use that $S_0=S\wedge \check{\omega}_\Lambda^{\dim\Lambda-1}$, $\omega_\Lambda^{\dim\Lambda+1}=0$ and $d^{n}\widehat{T}_\mathcal{f}=(\mathcal{f}^{n})^{*}\widehat{\omega}+dd^{c}(g\circ f^{n})$, where $g$ is a continuous $\widehat{\omega}$-psh function on $\mathcal{X}_\Lambda$. So
	\begin{align}\label{eq:lemme7}
	\|S_n\|_{\pi^{-1}(K_1)} & =\int_{\pi^{-1}(K_1)} S\wedge \check{\omega}_\Lambda^{\dim\Lambda-1}\wedge \left( d^{n}\widehat{T}_\mathcal{f}-dd^{c}(g\circ \mathcal{f}^{n})\right)^{k+1-p} \nonumber\\ 
	&\ +  (k+1-p)\int_{\pi^{-1}(K_1)} S\wedge \check{\omega}_\Lambda^{\dim\Lambda}\wedge \left( d^{n}\widehat{T}_\mathcal{f}-dd^{c}(g\circ \mathcal{f}^{n}) \right)^{k-p}.
	\end{align}
	Now, 
	\begin{align*} \left( d^{n}\widehat{T}_\mathcal{f}-dd^{c}(g\circ \mathcal{f}^{n})\right)^{k+1-p}&=\sum_{i=0}^{k+1-p} \binom{k+1-p}{i} (-1)^{k+1-p-i}d^{ni}\widehat{T}^i_\mathcal{f}\wedge \left(dd^{c}(g\circ \mathcal{f}^{n})\right)^{k+1-p-i}, \\
	\left( d^{n}\widehat{T}_\mathcal{f}-dd^{c}(g\circ \mathcal{f}^{n})\right)^{k-p}&=\sum_{j=0}^{k-p} \binom{k-p}{i} (-1)^{k-p-i}d^{ni}\widehat{T}^i_\mathcal{f}\wedge \left(dd^{c}(g\circ \mathcal{f}^{n})\right)^{k-p-i}.
	\end{align*}
	Take $i\in \{0,\dots,k-p\}$. 
	By the Chern-Levine-Nirenberg inequality, see \cite[Th\'eor\`eme~2.2]{Demailly_SMF} and \cite[\S 3.3 p. 146]{Demailly}, there exists a constant $C_1>0$ depending only on $K_1$, $K_2$, $\|\widehat{T}_\mathcal{f}\|_{\pi^{-1}(K_2)}$ and $\|g\|_{L^{\infty}(\pi^{-1}(K_2))}$ such that
	\begin{align*}\left|\int_{\pi^{-1}(K_1)} S\wedge \check{\omega}_\Lambda^{\dim\Lambda-1}\wedge  \widehat{T}^i_\mathcal{f}\wedge \left(dd^{c}(g\circ \mathcal{f}^{n})\right)^{k+1-p-i} \right| \leq C_1 \|S_0\|_{\pi^{-1}(K_2)},\\
	\left|\int_{\pi^{-1}(K_1)} S\wedge \check{\omega}_\Lambda^{\dim\Lambda}\wedge  \widehat{T}^i_\mathcal{f}\wedge \left(dd^{c}(g\circ \mathcal{f}^{n})\right)^{k-p-i} \right| \leq C_1 \|S_0\|_{\pi^{-1}(K_2)}.\end{align*}
	Now, subtracting $d^{(k+1-p)n}\int_{\pi^{-1}(K_1)}\widehat{T}_\mathcal{f}^{k+1-p}\wedge S\wedge\check{\omega}_\Lambda^{\dim\Lambda-1}$ from both sides of \eqref{eq:lemme7}, we have, up to enlarging $C_1$: 
		\begin{align*}
	\left|\|S_n\|_{\pi^{-1}(K_1)} - d^{(k+1-p)n}\int_{\pi^{-1}(K_1)}\widehat{T}_\mathcal{f}^{k+1-p}\wedge S\wedge\check{\omega}_\Lambda^{\dim\Lambda-1}\right| & \leq C_1\|S_0\|_{\pi^{-1}(K_2)}d^{(k-p)n}
	\end{align*}
	which ends the proof.
\end{proof}

\begin{proof}[Proof of Proposition~\ref{lm:bifmeasure}]
Let $\mathcal{K}$ be any compact subset of $\mathcal{X}_\Lambda$. Remark that it is always contained in a compact subset of the form $\pi^{-1}(K)$, where $K$ is compact in $\Lambda$, since the fibers of $\pi$ are compact. By Lemma~\ref{lm:growthmass}, we have $\|(\mathcal{f}^{n})_*(S)\wedge\check{\omega}_\Lambda^{\dim\Lambda-1}\|_{\pi^{-1}(K)}=O(d^{n(k-p)})$, if and only if $T_{f,S}\wedge\omega_\Lambda^{\dim\Lambda-1}$ is zero on $K$.

As this holds for any compact set $K\Subset \Lambda$, this proves the equivalence between points 2 and 3, and point 2 obviously implies point 1. The proof that point 1 implies point 3 follows also from Lemma~\ref{lm:growthmass}.  Indeed, if $T_{\mathcal{f},S}\neq 0$, then for a suitable compact set $K_1$, $\int_{\pi^{-1}(K_1)}\widehat{T}_\mathcal{f}^{k+1-p}\wedge S\wedge\check{\omega}_\Lambda^{\dim\Lambda-1}:=2C\neq 0$ and  $\|(\mathcal{f}^{n})_*(S)\wedge\check{\omega}_\Lambda^{\dim\Lambda-1}\|_{\pi^{-1}(K_1)}\geq C (d^{n(k+1- p)})$ for $n$ large enough.
\end{proof}

\begin{remark}\normalfont
We say that a subvariety $\mathcal{C}\subset\mathcal{X}_\Lambda$ of pure dimension $\dim(\Lambda)$ is horizontal if $\mathcal{C}\cap \pi^{-1}\{\lambda\}$ is finite for all $\lambda\in \Lambda$.

 When $S=[\mathcal{C}]$ is the current of integration on a horizontal variety of pure dimension $\dim(\Lambda)$, the definition of stability implies that, for any compact set $K$, $\|[f^n(\mathcal{C})]\|_K$ is bounded (Apply Proposition~\ref{lm:bifmeasure} (2) with $p=k$).
By a famous Theorem of Bishop (see~\cite[Corollary~p.205]{Chirka}), there exists a subsequence $(f^{n_k}(\mathcal{C}))_k$ which converges in Hausdorff topology towards an analytic set $\mathcal{C}_\infty$. In particular, if $\dim \Lambda=1$ and $\mathcal{C}$ is the graph of a holomorphic section $\sigma:\Lambda\to \mathcal{X}_\Lambda$ of $\pi$, this is equivalent to the local uniform convergence of the sequence $\sigma_k:=\mathcal{f}^{n_k}\circ \sigma$ of sections of $\pi$ to a holomorphic section. Indeed, if $\mathcal{C}_\infty$ had a vertical component $V$ over $t_0$, then we have $\widehat{T}_f\wedge [V]=T_{f_{t_0}}\wedge[V]\neq0$.
In other words, $(\mathcal{f}^{n}\circ \sigma)_n$ is a normal family. 
\end{remark}

\section{Algebraic families of polarized endomorphisms}\label{section3}
Let $\mathcal{B}$ be a normal complex projective variety and let $\mathbf{K}:=\mathbb{C}(\mathcal{B})$ be its field of rational functions.  Let $(X,f,L)$ be a polarized endomorphism over $\mathbf{K}$ with $\dim X =k$.
The purpose of this section is to relate bifurcation currents to height with explicit error terms and use this to prove Theorem~\ref{tm:formulaheight}.

\subsection{Polarized endomorphisms over a function field versus algebraic families}\label{sec:algebraic}

By Remark~\ref{rem_replace_very_ample}, we assume that  $(X,f,L)$ is a polarized endomorphism with $L$ very ample.  
By Proposition~\ref{prop:fak} applied to $(X,f,L)$, we have an embedding $i:X\hookrightarrow \p^N_\mathbf{K}$ with $L=i^*\mathcal{O}_{\p^N}(1)$ and an endomorphism $F:\p^N_\mathbf{K}\to\p^N_\mathbf{K}$ such that $i\circ F=f\circ i$. This endomorphism $F$ gives rise to a family $(\mathbb{P}^N_\mathbb{C}\times\mathcal{B},\mathcal{F},\mathcal{O}_{\p^N}(1))$ of endomorphisms of $\p^N_\C$ parametrized by $\mathcal{B}$.

Let $\mathcal{X}$ be the Zariski closure in $\p^N_\C\times \mathcal{B}$ of the image of $i(X)$ by the isomorphism between $\p^N_\mathbf{K}$ and the generic fiber of $\p^N_\mathbb{C}\times\mathcal{B}\to\mathcal{B}$. Let $\mathcal{L}$ be the restriction of $\mathcal{O}_{\p^N}(1)$ restricted to $\mathcal{X}$. Let also $\iota:\mathcal{X}\hookrightarrow \p^N_\C\times \mathcal{B}$ be the inclusion and $\mathcal{f}:=\mathcal{F}|_\mathcal{X}$. Then $(\mathcal{X},\mathcal{f},\mathcal{L})$ is a model for $(X,f,L)$.

The models $(\mathbb{P}^N_\mathbb{C}\times\mathcal{B},\mathcal{F},\mathcal{O}_{\p^N}(1))$ and $(\mathcal{X},\mathcal{f},\mathcal{L})$ also induce analytic families of polarized endomorphisms over a common regular part $\Lambda\subset\mathcal{B}$.

Conversely, an algebraic family of polarized endomorphisms gives rise to a polarized endomorphism $(X,f,L)$, with $L$ very ample. Indeed, let $X$ be the generic fiber of $\mathcal{X}$, and let $L$ and $f$ be the respective restrictions of $\mathcal{L}$ and $\mathcal{f}$ to $X$. Then $(X,f,L)$ is a polarized endomorphism over $\mathbf{K}$. Note that in what follow, when we speak of an algebraic family of polarized endomorphisms, we always mean that  the corresponding line bundle $L$ is very ample.

 We now illustrate the above notions in the \emph{Desboves family} $\mathcal{f}$; we will also explore it in Example~\ref{Desboves2}.  
	 We start with the action of $\mathcal{f}$ on several invariant sets and we explain why the family is not isotrivial (see Definition~\ref{def_isotrivial} in the introduction).

\begin{example}[The elementary Desboves family I]\normalfont \label{Desboves1}

This family is already used by e.g.\ \cite{Bonifant-Dabija,Bonifant-Dabija-Milnor} to product attractors and in \cite{Bianchi-Taflin} to construct an open set of bifurcation. For any $\lambda\in\C^*$, let $f_\lambda:\mathbb{P}^2\to\mathbb{P}^2$ be the endomorphism  of degree $4$ given by
\[f_\lambda\left([x:y:z]\right):=[-x(x^3+2z^3):y(z^3-x^3+\lambda(x^3+y^3+z^3)):z(2x^3+z^3)].\]
This defines an algebraic family $(\mathbb{P}^2\times\mathbb{P}^1,\mathcal{f},\mathcal{O}_{\mathbb{P}^2}(1))$ with regular part $\mathbb{C}^*\subset\mathbb{P}^1$.
This family induces an endomorphism of degree $4$ of $\mathbb{P}^2_{\mathbb{C}(z)}$.

 For any $\lambda\in\mathbb{C}^*$, the point $\rho_0:=[0:1:0]$ is totally invariant by $f_\lambda$, i.e.\  $f_\lambda^{-1}\{\rho_0\}=\{\rho_0\}$ and $f_\lambda$ preserves the pencil $\mathcal{P}$ of lines of $\mathbb{P}^2$ passing through $\rho_0$. Furthermore, $f_\lambda$ preserves the lines $H_x=\{x=0\}$ and $H_z=\{z=0\}$ which belong to $\mathcal{P}$, the line $H_y=\{y=0\}$, and the Fermat curve
\[\mathcal{C}:=\{[x:y:z]\in\mathbb{P}^2\, : \ x^3+y^3+z^3=0\}.\]
Moreover, for any $\lambda\in\C^*$, the restriction of $f_\lambda$ to each of those curves can be described:
\begin{itemize}
\item the restriction of $f_\lambda$ to the line $H_x$ is the degree $4$ polynomial $p_\lambda(z)=\lambda z^4+(1+\lambda)z$,
\item the restriction of $f_\lambda$ to the line $H_z$ is the degree $4$ polynomial $p_{-\lambda}$,
\item the restriction of $f_\lambda$ to the line $H_y$ is the Latt\`es map
\[g:[x:y]\in\mathbb{P}^1\longmapsto [-x(x- 3+2y^3):y^3(2x^3+y^3)]\in\mathbb{P}^1.\]
this defines a constant, hence isotrivial, family, 
\item  the restriction of $f_\lambda$ to the elliptic curve $\mathcal{C}$ is the isogeny $u\mapsto -2u+\beta$ for some $\beta$ independent of $\lambda$ so we have again isotriviality.
\end{itemize}
Note that the family $(\mathbb{P}^2\times\mathbb{P}^1,\mathcal{f},\mathcal{O}_{\mathbb{P}^2}(1))$ is obviously non-isotrivial, since the restriction $p_\lambda$ of $f_\lambda$ to $X$ has a fixed point with multiplier $1+\lambda$, which is a non-constant rational function of $\lambda$.
\end{example}

\medskip

We will regularly use the following lemma, which says that isotriviality can be read on any iterate and does not depend on the chosen regular part.

\begin{lemma}\label{lm:isotriviality-iterate}
Let $(\mathcal{X},f,\mathcal{L})$ be an algebraic family of polarized endomorphisms of degree $d$. Let $\Lambda$ and $\Lambda'$ be two regular parts for $(\mathcal{X},f,\mathcal{L})$. The following are equivalent
\begin{enumerate}
\item $(\mathcal{X},\mathcal{f},\mathcal{L})$ is isotrivial over $\Lambda$,
\item $(\mathcal{X},\mathcal{f},\mathcal{L})$ is isotrivial over $\Lambda'$,
\item there exists $n\geq1$ such that  $(\mathcal{X},\mathcal{f}^n,\mathcal{L})$ is isotrivial over $\Lambda$.
\end{enumerate}
\end{lemma}

\begin{proof}
If $(\mathcal{X},\mathcal{f},\mathcal{L})$ is isotrivial over $\Lambda$, obviously $(\mathcal{X},\mathcal{f}^n,\mathcal{L})$ is isotrivial over $\Lambda$ for any $n\geq1$. We thus prove the converse implication. Let $X$ be the generic fiber of $\pi$, $L:=\mathcal{L}|_X$ and $f:=\mathcal{f}|_X$. 

Assume the family $(\mathcal{X},\mathcal{f}^n,\mathcal{L})$ is isotrivial over $\Lambda$, whence up to base change and conjugacy over $\C(\mathcal{B})$, the polarized endomorphism $(X,f^n,L)$ is defined over $\C$ whence $\mathcal{X} = X_\C \times \mathcal{B}$. Fix $t_0 \in \Lambda$ and let $\Per_q(f_{t_0}^n)$ be the finite set of periodic points of $f_{t_0}^n$ of period $q$. For any $Q$, the set $E_Q$ of endomorphisms of degree $d$ of $X_\C$ that coincide on $\bigcup_{q\leq Q}\mathrm{Per}_q(f_{t_0}^n)$ with $f_{t_0}$ is Zariski closed. As $\bigcup_q\mathrm{Per}_q(f_{t_0}^n)$ is Zariski dense by Theorem~\ref{prop:endopolarized}, the set $\bigcap_Q E_Q$ is reduced to $\{f_{t_0}\}$. By noetherianity, this implies $E_Q=\{f_{t_0}\}$ for a large enough $Q$.  Observe now that 
for all $q\leq Q$, all $x \in \Per_q(f_{t_0}^n)$ and all $t\in \Lambda$, one has $f_t(x) \in \Per_q(f_{t_0}^n)$ since  $f^{nq}_{t}(f_t(x))=f_t(x)$. In particular, there are only finitely many choices for $f_t$ and a continuity argument implies that $f$ is again constant.

To conclude, note that the variety $X$ and the line bundle $L$ have to be isotrivial themselves by assumption and that $(X,f,L)$ is a polarized endomorphism which is isotrivial.

The proof implies the independence on the regular part. 
\end{proof}

To finish the present discussion, we recall that to any subvariety $Z$ of $X$, which is defined over $\mathbf{K}$, we can associate a subvariety $\mathcal{Z}$ of $\mathcal{X}$ (which can be defined as the Zariski closure of $Z$ in $\mathcal{X}$) such that the restriction of $\pi$ to $\mathcal{Z}$ is flat over $\Lambda$, restricting $\Lambda$ if necessary.
In particular, to any point $\mathrm{x}\in X(\mathbf{K})$ with corresponding subvariety $\mathcal{x}$, we can associate a rational section $\sigma:\mathcal{B}\dashrightarrow\mathcal{X}$ of $\pi$, i.e.\  a rational map such that
\begin{enumerate}
\item $\pi\circ\sigma=\mathrm{id}_\mathcal{B}$,
\item $\sigma$ is regular over $\Lambda$ (restricting $\Lambda$ is necessary),
\item the Zariski closure of $\sigma(\Lambda)$ in $\mathcal{X}$ is $\mathcal{x}$.
\end{enumerate}

\subsection{Global properties of the bifurcation current}\label{sec:globalproperties}
 Let $(\mathcal{X},\mathcal{f},\mathcal{L})$ be a model of $(X,f,L)$ as above, with $L$ very ample. 

\smallskip

We use the notations of Theorem~\ref{tm:formulaheight}. Let $\widehat{\omega}$ be a closed positive form on $\mathcal{X}$ cohomologous to $c_1(\mathcal{L})$ and let 
 $\check{\omega}_\mathcal{B}:=\pi^*(\omega_{\mathcal{B}})$.  The closed positive $(1,1)$-form $\widehat{\omega}+\check{\omega}_\mathcal{B}$ is a K\"ahler form on $\mathcal{X}$ cohomologous to $c_1(\mathcal{M})$. Observe that replacing $L$ with a tensor product $L^{\otimes e}$ will multiply the different quantities by the same power of $e$ (namely $e^{\ell +1}$) so we can assume that $L$ is very ample and this hypothesis does not affect the statements below (see Remark~\ref{rem_replace_very_ample}).

~

Fix a regular part $\Lambda$ and let $S$ be a positive closed $(p,p)$ current on $\mathcal{X}_\Lambda$. For any Borel subset $\Omega$ of $\mathcal{X}_\Lambda$, we let
\begin{align*}
\|S\|_{\Omega} :=& \left\langle S, \mathbf{1}_\Omega(\check{\omega}_\mathcal{B}+\widehat{\omega})^{k+\dim\mathcal{B}-p}\right\rangle,
\end{align*}
where $\langle \cdot,\cdot \rangle$ is the duality bracket between currents and forms. 
When $\|S\|:=\|S\|_{\mathcal{X}_\Lambda}<+\infty$, the current $S$ extends trivially to the Zariski closure $\mathcal{X}$ of $\mathcal{X}_\Lambda$ as a closed positive $(p,p)$-current on $\mathcal{X}$ that we still denote by $S$ (\cite[Chapter III, Theorem (2.3)]{Demailly}).
Note that the fact that $\|S\|_{\mathcal{X}_\Lambda}$ is finite does not depend on the choice of regular part (if $\|S\|_{\mathcal{X}_{\Lambda'}}<+\infty$ for some $\Lambda'\subset\Lambda$, then so is $\|S\|_{\mathcal{X}_{\Lambda}}$).

\medskip

We use the notations of Section~\ref{sec:algebraic}. The following is a global version of Lemma~\ref{lm:growthmass}:
\begin{proposition}\label{Stable=bounded}
Let $(\mathcal{X},\mathcal{f},\mathcal{L})$ be an algebraic family of polarized endomorphisms. Let $\Lambda$ be a regular part and let $k:=\dim X_\lambda$ for any $\lambda\in \Lambda$. Then there exists a constant $C>0$ depending only on $(\mathcal{X},\mathcal{f},\mathcal{L})$ such that 
\begin{enumerate}
	\item  For any closed positive $(p,p)$-current $S$ on $\mathcal{X}$ with $0\leq p \leq \dim\mathcal{X}-1$ and any $n\geq1$,
	\[\left|\|\widehat{T}_\mathcal{f}\wedge S\|-d^{-n}\|(\mathcal{f}^n)^*(\widehat{\omega})\wedge S\|\right|\leq C d^{-n}\|S\|.\]
	\item For any $1\leq p \leq k$, any $0\leq r\leq \dim\mathcal{B}$ and any  closed positive $(p,p)$-current $S$ on $\mathcal{X}$ and any $n\geq1$, if $q:=k+\dim\mathcal{B}-r-p$, 
	\[\left|\int_{\mathcal{X}_\Lambda}(f^{n})_*(S)\wedge\widehat{\omega}^{q}\wedge\check{\omega}_\mathcal{B}^{r}-d^{nq}\int_{\mathcal{X}_\Lambda} S\wedge\widehat{T}_\mathcal{f}^{q}\wedge  \check{\omega}_\mathcal{B}^{r}\right|\leq C\sum\limits_{j<q} d^{nj}\|S\wedge\widehat{T}_\mathcal{f}^{j}\wedge  (\mathcal{f}^n)^*(\widehat{\omega})^{q-j-1}\wedge  \check{\omega}_\mathcal{B}^{r}\|.\] 
\end{enumerate}

\end{proposition}

\begin{proof}
Let us consider first the case where $\Lambda$ is affine. We thus can define an embedding $\iota_1:\mathcal{B}\hookrightarrow\p^{M}$ with $\iota_1^{-1}(\mathbb{A}^{M}(\mathbb{C}))=\Lambda$. We already embedded $\mathcal{X}$ in $\p^N\times \mathcal{B}$ with $\mathcal{L}=\iota^*(\mathcal{O}_{\p^N}(1))$:
\[\mathcal{X}\stackrel{\iota}{\hookrightarrow} \p^N\times \mathcal{B}\stackrel{(\mathrm{id},\iota_1)}{\hookrightarrow} \p^N\times \p^M,\]
We now write $(z,t)$ for a point of $\p^N\times\p^M$ and let $\|\cdot\|$ be the standard Hermitian norm on $\mathbb{A}^M(\C)$. 
We also identify $\lambda=\pi({ x})$ with the second coordinate $t$ of $(\mathrm{id},\iota_1)\circ\iota({ x})$.
\begin{lemma}\label{goodgrowth}
There exist $C_1,C_2>0$ such that for all $n\geq1$, we can write
\[ \frac{1}{d^{n}}(\mathcal{f}^n)^*(\widehat{\omega}) - \widehat{T}_\mathcal{f} = \mathrm{dd}^c \phi_n \quad \text{on} \ \mathcal{X}_\Lambda, \]
where  $\phi_n:\mathcal{X}_\Lambda\to\mathbb{R}$ is a function such that $|\phi_n(z)| \leq d^{-n}(C_1 \log^{+}\|\lambda\|+C_2)$, for all $\lambda\in\Lambda$ and all $z\in X_\lambda$. 
\end{lemma}

Lemma~\ref{goodgrowth} will be proved below.
We take the lemma for granted and continue the proof. Following \cite[p. 381, Proof of Theorem 1.6]{GOV}, we take any closed positive $(p,p)$-current $S$ on $\mathcal{X}$ with $0\leq p \leq \dim\mathcal{X}-1$ and any $n\geq1$. For any $A>0$, we define the following test function 
\[
\Psi_A(\lambda):= \frac{  \log\max(\|\lambda\|, e^{2A})-\log\max(\|\lambda\|, e^A)}{A}.
\]
For any $R>0$, the current $\mathrm{dd}^c\log\max(\|\lambda\|, R)$ has mass $1$ on $\mathbb{C}^M$ for the Fubini-Study metric by Lelong-Poincar\'e formula, whence it restricts to $\Lambda$ as a current of mass $\deg_\mathcal{N}(\mathcal{B})$. Thus, $\Psi_A$ is continuous and DSH on $\Lambda$ with $\mathrm{dd}^c \Psi_A=T_A^+-T_A^-$
where $T^\pm_A$ are some positive closed $(1,1)$-currents whose masses are finite with $\|T^\pm_A \|\leq C'/A$ for some $C'>0$ depending neither on $A$ nor on $T^\pm_A$. Observe also that $\Psi_A$ is equal to $1$ in $B(0,e^A)$, and $0$ outside $B(0,e^{2A})$. We first prove point 1; pick $n\geq1$, then then, since $S$ is closed, Stokes theorem provides
\begin{align*}
J_n^A & := \left\langle\left(\widehat{T}_\mathcal{f}-\frac{1}{d^n}(\mathcal{f}^n)^*(\widehat{\omega})\right)\wedge \left(\widehat{\omega}+\check{\omega}_\mathcal{B}\right)^{k+\dim\mathcal{B}-p-1}\wedge  S,\Psi_A\circ \pi\right\rangle\\
& = \left\langle \phi_n\cdot \left(\widehat{\omega}+\check{\omega}_\mathcal{B}\right)^{k+\dim\mathcal{B}-p-1},\mathrm{dd}^c (\Psi_A\circ \pi)\wedge  S \right\rangle
\end{align*}
and the definition of $\Psi_A$ implies
\begin{align*}
|J_n^A| & \leq \int_{\iota^{-1}(\p^N\times B(0,e^{2A}))}|\phi_n| \left(\widehat{\omega}+ \check{\omega}_\mathcal{B}\right)^{k+\dim\mathcal{B}-p-1}\wedge \pi^*(T_A^{+}+T_A^{-})\wedge  S \\
&\leq \sup_{\p^N\times B(0,e^{2A})}|\phi_n| \int_{\mathcal{X}} \left(\widehat{\omega}+ \check{\omega}_\mathcal{B}\right)^{k+\dim\mathcal{B}-p-1}\wedge \pi^*(T_A^{+}+T_A^{-})\wedge  S \\
&  \leq\frac{C_3}{A}\sup_{\p^N\times B(0,e^{2A})}|\phi_n|\cdot\|S\|,
\end{align*}
for some constant $C_3>0$ since the mass of the wedge product can be bounded in cohomology, up to a multiplicative constant,  by the product of the masses. Lemma~\ref{goodgrowth} then implies $|J_n^A|\leq C_4d^{-n}\|S\|$ for some constant $C_4>0$ which is independent of $A$. Making $A\to\infty$ gives the first point of the proposition.

\medskip

We now prove the second estimate of the Proposition. Let $q:=k+\dim\mathcal{B}-r-p$, since $\widehat{T}_\mathcal{f}=\frac{1}{d^{n}}(\mathcal{f}^{n})^{*}(\widehat{\omega})-\mathrm{dd}^c\phi_n$,
a direct computation gives
\begin{align*} \widehat{T}_\mathcal{f}^{q}-\left(\frac{1}{d^{n}}(\mathcal{f}^{n})^{*}\widehat{\omega}\right)^q &= \left( \widehat{T}_\mathcal{f}-\frac{1}{d^{n}}(\mathcal{f}^{n})^{*}\widehat{\omega}\right) \wedge \sum_{s=0}^{q-1}  \widehat{T}_\mathcal{f}^{s} \wedge  \left(\frac{1}{d^{n}}(\mathcal{f}^{n})^{*}\widehat{\omega}\right)^{q-1-s} \\
& = \sum_{s=0}^{q-1}  \mathrm{dd}^c \phi_n \wedge  \widehat{T}_\mathcal{f}^{s} \wedge  \left(\frac{1}{d^{n}}(\mathcal{f}^{n})^{*}\widehat{\omega}\right)^{q-1-s}. \end{align*}
Since $ \pi\circ \mathcal{f}^{n}= \pi$, 
\begin{align*}
I^A_n & := \left\langle\frac{1}{d^{qn}}(\mathcal{f}^{n})_*(S)\wedge\left(\widehat{T}_\mathcal{f}^{q}-\widehat{\omega}^{q}\right)\wedge \check{\omega}_\mathcal{B}^{r} ,\Psi_A\circ \pi\right\rangle\\
& =\left\langle S\wedge \left(\widehat{T}_\mathcal{f}^{q}-\left(\frac{1}{d^{n}}(\mathcal{f}^{n})^{*}\widehat{\omega}\right)^{q}\right),\Psi_A\circ \pi\circ f^{n}\cdot  ((\pi\circ \mathcal{f}^n)^*\omega_\mathcal{B})^{r} \right\rangle\\
& = \sum_{s=0}^{q-1}d^{-n(q-1-s)} \left\langle S\wedge(\mathrm{dd}^c\phi_n)\wedge \widehat{T}_\mathcal{f}^{s}\wedge(\mathcal{f}^n)^*(\widehat{\omega})^{q-1-s},\Psi_A\circ \pi\cdot  \check{\omega}_\mathcal{B}^{r} \right\rangle\\
& = \sum_{s=0}^{q-1}d^{-n(q-1-s)}\int_\mathcal{X} \phi_n\cdot S\wedge \widehat{T}_\mathcal{f}^{s}\wedge(\mathcal{f}^n)^*(\widehat{\omega})^{q-1-s}\wedge \mathrm{dd}^c(\Psi_A\circ \pi)\wedge  \check{\omega}_\mathcal{B}^{r} ,
\end{align*}
where we used  Stokes formula.

\smallskip

We now let $S_{s}:=S\wedge \widehat{T}_\mathcal{f}^{s}\wedge\check{\omega}_\mathcal{B}^{r}$ for any $0\leq s\leq q-1$. The above implies
\begin{align*}
|I_n^A| & \leq \sum_{s=0}^{q-1}d^{-n(q-1-s)}\int_{\mathcal{X}} |\phi_n|\cdot S_s\wedge (\mathcal{f}^n)^*(\widehat{\omega})^{q-1-s}\wedge \pi^*(T_A^++T_A^-)\\
 & \leq \sup_{\p^N\times B(0,e^{2A})}|\phi_n|\sum_{s=0}^{q-1}d^{-n(q-1-s)}\int_{\mathcal{X}} S_s\wedge (\mathcal{f}^n)^*(\widehat{\omega})^{q-1-s}\wedge \pi^*(T_A^++T_A^-)\\
& \leq \sum_{s=0}^{q-1}d^{-n(q-1-s)}\frac{C_3}{A}\cdot \| S_s\wedge \left((\mathcal{f}^n)^*(\widehat{\omega})\right)^{q-1-s}\|\cdot \sup_{\p^N\times B(0,e^{2A})}|\phi_n|,
\end{align*}
for some constant $C_5>0$, since again, the mass of the wedge product can be bounded in cohomology, up to a multiplicative constant,  by the product of the masses. Using again Lemma \ref{goodgrowth}, we find a constant $C_6>0$ such that
\begin{align*}
|I_n^A| & \leq C_6\sum_{s=0}^{q-1}d^{-n(q-s)}\cdot \| S_s\wedge \left((\mathcal{f}^n)^*(\widehat{\omega})\right)^{q-1-s}\|.
\end{align*}
Since the constant does not depend on $A$, we can make $A\to\infty$ and multiply by $d^{nq}$ to complete the proof. 

We now explain how to deduce the result for the case of the maximal regular part $\Lambda_{\max}$. Let $\Lambda \subset \Lambda_{\max}$ be an affine Zariski dense open subset of $\Lambda_{\max}$. Observe that $\widehat{\omega}$, $\check{\omega}_\mathcal{B}$ have continuous potentials, and $\widehat{T}_\mathcal{f}$ has continuous potentials on $\Lambda_{\max}\backslash \Lambda$. In particular, all the currents $(f^{n})_*(S)\wedge\widehat{\omega}^{q}\wedge\check{\omega}_\mathcal{B}^{r}$,  $\widehat{T}_\mathcal{f}^{q}\wedge S\wedge \check{\omega}_\mathcal{B}^{r}$, 
$\widehat{T}_\mathcal{f}^{s}\wedge S\wedge (\mathcal{f}^n)^*(\widehat{\omega})^{q-1-s}\wedge  \check{\omega}_\mathcal{B}^{r}$ give no mass to the analytic set $\pi^{-1}(\Lambda_{\max}\backslash \Lambda)$. Hence, the inequalities proved for $\Lambda$ affine also stand for $\Lambda_{\max}$ and we can deduce from that case the proposition for any regular part.
\end{proof}

We now give the proof of Lemma~\ref{goodgrowth}.

\begin{proof}[Proof of  Lemma \ref{goodgrowth}] Let $F:\p^N\times\Lambda\to\p^N\times\Lambda$ be the family induced by $(\mathcal{X},\mathcal{f},\mathcal{L})$.
As $H^1(\operatorname{End}_d(\p^N), \R)=\{0\}$ by \cite[Lemma 4.9]{BB1}, there is a family of polynomial lifts $\tilde H:\C^{N+1}\times \operatorname{End}_d(\p^N)\to\C^{N+1}\times \operatorname{End}_d(\p^N)$ which induces the universal family $h:\p^N\times \operatorname{End}_d(\p^N)\to \p^N\times \operatorname{End}_d(\p^N)$. Now, the family $F:\p^N\times\Lambda\to\p^N\times\Lambda$ corresponds to a morphism $\psi:\Lambda \to\operatorname{End}_d(\p^N)$ and we can define a family of homogeneous polynomial lifts $\tilde F: \C^{N+1}\times \Lambda\to\C^{N+1}\times \Lambda$ by letting $\tilde F_\lambda:=\tilde H_{\psi(\lambda)}$, for all $\lambda\in \Lambda$.

We thus can write $\tilde{F}_\lambda=(P_{0,\lambda},\ldots,P_{N,\lambda}):\C^{N+1}\to\C^{N+1}$, for any $\lambda\in\Lambda$, where $P_{i,\lambda}\in\C[\Lambda][z_0,\ldots,z_N]$ are homogeneous polynomials and $\mathrm{Res}(P_{0,\lambda},\ldots,P_{N,\lambda})\in \C^*$ if $\lambda\in\Lambda$. Since $p\mapsto\frac{1}{d}\log\|\tilde{F}_\lambda(p)\|-\log\|p\|$ is $0$-homogeneous, it induces a function $g_0:\p^N(\C)\times\Lambda\to\R$ defined, for $(z,\lambda)\in\p^N(\C)\times\Lambda$, by $g_0(z,\lambda):=\frac{1}{d}\log\|\tilde{F}_\lambda(p)\|-\log\|p\|$, for any $p=(p_0,\ldots,p_N)\in\C^{N+1}-\{0\}$ such that $z=[p_0:\cdots:p_N]$. We have
$d^{-1}{F}^*\hat{\omega}_{\mathrm{FS}}-\hat{\omega}_{\mathrm{FS}}=\mathrm{dd}^cg_0$.
By construction, 
\begin{align}
\widehat{T}_{F}-\frac{1}{d^n}(F^n)^*\hat{\omega}_{\mathrm{FS}}=\mathrm{dd}^c\left(\sum_{j=n}^\infty\frac{1}{d^j}g_0\circ{F}^j \right).\label{estimatedegenerate}
\end{align}
Since the coefficients of $\tilde{F}_\lambda$ in a given system of homogeneous coordinates are in $\C[\Lambda]$, then $\mathrm{Res}(\tilde{F}_\lambda)\in \C[\Lambda]$, so that 
\[\left|\log|\mathrm{Res}(\tilde{F}_\lambda)|\right|\leq C_1\log^+\|\lambda\|+C_2.\]
An application of the homogeneous Hilbert's Nullstellensatz (see, e.g., \cite[proof of Lemma 6.5]{GOV}) implies there exist a constant $C>0$ and an integer $N\geq1$ such that
\[C^{-1}|\mathrm{Res}(\tilde{F}_\lambda)|\leq \frac{\|\tilde{F}_\lambda(p)\|}{\|p\|^d}\leq C \max\left(\|\lambda\|,1\right)^{\kappa}\quad \text{for  any} \ p\in\C^{N+1}\setminus\{0\} \ \text{and  any} \ \lambda\in\Lambda.\]
In particular, since $g_0(z,\lambda)=\log\|\tilde{F}_\lambda(p)\|/\|p\|^d$, this gives easily $|g_0(z,\lambda)|\leq \kappa \log^+\|\lambda\|+\log(C)$ for any $z\in\p^k$ and any $\lambda\in\Lambda$.

\medskip

Using that $\iota\circ \mathcal{f}=F\circ\iota$, we deduce the lemma from the above.
\end{proof}

\subsection{Global height function versus mass of a current}\label{sec:height}

Our choice of an ample line bundle $\mathcal{N}$ on $\mathcal{B}$ provides a naive height $h_{X,L}$ on $X(\overline{\mathbf{K}})$: For any irreducible subvariety $Z$ of $X$ defined over $\mathbf{K}$, of dimension $0\leq \ell\leq k=\dim X$,
the height $h_{X,L}(Z)$ can be defined as the intersection number 
\[h_{X,L}(Z):=\left(\mathcal{Z}\cdot c_1(\mathcal{L})^{\ell+1}\cdot c_1(\pi^*\mathcal{N})^{\dim\mathcal{B}-1}\right),\]
where $\mathcal{Z}$ is the Zariski closure of $Z$ in $\mathcal{X}$, see, e.g., \cite{Gubler,Gubler2,Faber,Cantat-Gao-Habegger-Xie}. As $\mathcal{Z}$ is irreducible, one has $\dim(\mathcal{Z}\cap (\mathcal{X}\setminus\mathcal{X}_\Lambda))<\dim\mathcal{Z}$ for any regular part $\Lambda$ and since $\check{\omega}_\mathcal{B}$ and $\widehat{\omega}$ have continuous potentials, the height $h_{X,L}(Z)$ of $Z$ can be computed as
\begin{align*}
h_{X,L}(Z) =\int_{\mathcal{X}}\widehat{\omega}^{\ell+1}\wedge \check{\omega}_\mathcal{B}^{\dim\mathcal{B}-1}\wedge[\mathcal{Z}]=\int_{\mathcal{X}_\Lambda}\widehat{\omega}^{\ell+1}\wedge \check{\omega}_\mathcal{B}^{\dim\mathcal{B}-1}\wedge[\mathcal{Z}],
\end{align*}
independently of the regular part. When $Z$ is defined over a finite extension $\mathbf{K}'$ of $\mathbf{K}$, we let $\rho':\mathcal{B}'\to\mathcal{B}$ be the normalization of $\mathcal{B}$ in $\mathbf{K}'$. If $\mathcal{X}':=\mathcal{X}\times_\mathcal{B}\mathcal{B}'$, the projection $\rho:\mathcal{X}'\to\mathcal{X}$ onto the first factor is a finite branched cover and, if $\mathcal{Z}'$ is the Zariski closure of $Z$ in $\mathcal{X}'$, we can set
\begin{align*}
h_{X,L}(Z) & =\frac{1}{[\mathbf{K}':\mathbf{K}]}\left(\mathcal{Z}'\cdot c_1(\rho^*\mathcal{L})^{\ell+1}\cdot c_1(\rho^*\pi^*\mathcal{N})^{\dim\mathcal{B}-1}\right).
\end{align*}
Up to base change, we thus can assume in the sequel that $Z$ is defined over $\mathbf{K}$.

\medskip

By the functorial properties of Weil heights, and since $f^*L\simeq L^{\otimes d}$, we have $h_{X,L}\circ f=d\cdot h_{X,L}+O(1)$
on $X(\bar{\mathbf{K}})$ (see, e.g.\ \cite{bombieri-gubler}). Following \cite{CS-height}, we define the \emph{canonical height} of $f$ as $\widehat{h}_f:=\lim_{n\to\infty}\frac{1}{d^n}h_{X,L}\circ f^n$.
By \cite[Theorem 1.1]{CS-height}, it is the unique function $\hat{h}_f:X(\bar{\mathbf{K}})\to\R_+$ satisfying
\begin{enumerate}
\item $\widehat{h}_f\circ f=d\cdot \widehat{h}_f$,
\item $\widehat{h}_f=h_{X,L}+O(1)$.
\end{enumerate}
Note that $\widehat{h}_f$ denotes both the function on $X(\bar{\mathbf{K}})$ and the canonical height on subvarieties in $X_{\bar{\mathbf{K}}}$ (of course, for points, that is zero dimensional variety, both definitions coincide).

\begin{proof}[Proof of Theorem~\ref{tm:formulaheight}]
Fix a regular part $\Lambda$. One can write
\[h_{X,L}(f^n(Z))=\int_{\mathcal{X}_\Lambda}(\mathcal{f}^n)_*[\mathcal{Z}]\wedge \widehat{\omega}^{\ell+1}\wedge \check{\omega}_\mathcal{B}^{\dim\mathcal{B}-1}.\]
The second point of Proposition~\ref{Stable=bounded} with $p= k- \ell$, $r= \dim(\mathcal{B})-1$ and thus $q=k+\dim \mathcal{B}-r-p=\ell +1$ gives 
\[h_{X,L}(f^n(Z))= d^{n (\ell+1)}\int_{\mathcal{X}_\Lambda} [\mathcal{Z}]\wedge\widehat{T}_\mathcal{f}^{\ell +1}\wedge  \check{\omega}_\mathcal{B}^{\dim{\mathcal{B}}-1} + O(d^{n\ell}). \]
In particular,  
\begin{align*}
\widehat{h}_{f}(Z) & :=\lim_{n\to\infty}\frac{1}{d^{n(\ell+1)}}h_{X,L}((f^n)_*(Z))
\end{align*}
is well defined and
\begin{align*}
\widehat{h}_{f}(Z) = \int_{\mathcal{X}_\Lambda} [\mathcal{Z}]\wedge\widehat{T}_\mathcal{f}^{\ell +1}\wedge \check{\omega}_\mathcal{B}^{\dim{\mathcal{B}}-1}.
\end{align*}
\end{proof}

When $Z$ is defined over a finite extension $\mathbf{K}'$ of $\mathbf{K}$, recall we denoted $\rho':\mathcal{B}'\to\mathcal{B}$ the normalization of $\mathcal{B}$ in $\mathbf{K}'$,  $\mathcal{X}':=\mathcal{X}\times_\mathcal{B}\mathcal{B}'$, the projection $\rho:\mathcal{X}'\to\mathcal{X}$ onto the first factor is a finite branched cover and,  $\mathcal{Z}'$ the Zariski closure of $Z$ in $\mathcal{X}'$.
We also let $(\mathcal{X}',\mathcal{f'},\rho^*\mathcal{L})$ the corresponding model with  $ \rho \circ \mathcal{f'}= \mathcal{f}\circ \rho $. In particular, $\widehat{T}_\mathcal{f'}= \rho^*(\widehat{T}_\mathcal{f})$.
Hence, by the above
\begin{align*}
	\widehat{h}_{f}(Z) & =\frac{1}{[\mathbf{K}':\mathbf{K}]} \int_{\mathcal{X}'_\Lambda} [\mathcal{Z}']\wedge\widehat{T}_\mathcal{f'}^{\ell +1}\wedge \check{\omega}_\mathcal{B'}^{\dim{\mathcal{B'}}-1}\\
	&= \frac{1}{[\mathbf{K}':\mathbf{K}]} \int_{\mathcal{X}'_\Lambda} [\mathcal{Z}']\wedge \rho^*(\widehat{T}_\mathcal{f})^{\ell +1}\wedge  \rho^*(\check{\omega}_\mathcal{B})^{\dim{\mathcal{B}}-1}
\end{align*}
since $\dim{\mathcal{B'}}= \dim{\mathcal{B}}$, $\check{\omega}_\mathcal{B'}= \rho^*(\check{\omega}_\mathcal{B'})$, and $\widehat{T}_\mathcal{f'}=\rho^*(\widehat{T}_\mathcal{f})$. Pushing forward to $\mathcal{X}$
\begin{align*}
	\widehat{h}_{f}(Z) 	&= \frac{1}{[\mathbf{K}':\mathbf{K}]} \int_{\mathcal{X}_\Lambda} \rho_*([\mathcal{Z}'])\wedge \widehat{T}_\mathcal{f}^{\ell +1}\wedge  \check{\omega}_\mathcal{B}^{\dim{\mathcal{B}}-1}
\end{align*}
As $\rho_*([\mathcal{Z}'])=[\mathbf{K}':\mathbf{K}] . [\mathcal{Z}]$, the result follows for $Z$ defined over a finite extension $\mathbf{K}'$.

\begin{remark}\normalfont
Remark that $\widehat{h}_{f}(f_*(Z))=d^{\ell+1}\cdot \widehat{h}_{f}(Z)$ where $\ell= \dim Z$. In particular, if $Z$ is preperiodic, i.e.\ if there exists $n>m\geq0$ such that $f^n(Z)=f^m(Z)$, we have $\widehat{h}_{f}(Z)=0$. Note that if $f(Z)=Z$ we have $\mathcal{f}_*[\mathcal{Z}]= d^{\ell} [\mathcal{Z}] $ in the sense of currents, since in any fiber $X_\lambda\cap \mathcal{Z}$, $f_\lambda$ is a polarized endomorphism of topological degree $d^\ell$ (see Lemma~\ref{lm:finite}). 
\end{remark}

Even though we do not need it in the sequel, we can prove the following statement, compare with \cite[Corollary 3.3]{Cantat-Gao-Habegger-Xie}.
\begin{corollary} Under the hypothesis of Theorem~\ref{tm:formulaheight},
$\widehat{h}_f(Z)=0$ if and only if for any closed positive $(1,1)$-current $\nu$ on $\mathcal{B}$ with continuous potentials we have
\[\int_{\mathcal{X}_\Lambda}\widehat{T}_\mathcal{f}^{\ell+1}\wedge[\mathcal{Z}]\wedge(\pi^*\nu)^{\dim\mathcal{B}-1}=0.\]
\end{corollary}
\begin{proof}
Theorem~\ref{tm:formulaheight} implies $\widehat{h}_f(Z)=0$ if and only if
\begin{align}
\int_{\mathcal{X}_\Lambda}\widehat{T}_\mathcal{f}^{\ell+1}\wedge[\mathcal{Z}]\wedge(\pi^*\omega_\mathcal{B})^{\dim\mathcal{B}-1}=0.\label{eq:caseomega}
\end{align}
It is thus sufficient to prove that if equation \eqref{eq:caseomega} holds, then
\begin{align}
\int_{\pi^{-1}(U)}\widehat{T}_\mathcal{f}^{\ell+1}\wedge[\mathcal{Z}]\wedge(\pi^*\nu)^{\dim\mathcal{B}-1}=0\label{eq:caseomega2}
\end{align}
 for any  closed positive $(1,1)$-current $\nu$ on $\mathcal{B}$ with continuous potentials and any open subset $U$ of $\Lambda$ which is relatively compact in $\Lambda$. It is sufficient to consider the case where $\nu$ is smooth since, by Richberg's theorem~\cite[\S 5.21, p. 43]{Demailly}, any continuous psh function can be locally approximated by smooth psh functions and the wedge product in \eqref{eq:caseomega2} is continuous with respect to $\nu$. Pick such a smooth current $\nu$. As $\omega_\mathcal{B}$ is strictly positive on $U$, there exists $C>0$ such that $C\omega_\mathcal{B}\geq \nu\geq0$ in the weak sense of currents. In particular,
\begin{align*}
0\leq \int_{\pi^{-1}(U)}\widehat{T}_\mathcal{f}^{\ell+1}\wedge[\mathcal{Z}]\wedge(\pi^*\nu)^{\dim\mathcal{B}-1}\leq \int_{\mathcal{X}_\Lambda}\widehat{T}_\mathcal{f}^{\ell+1}\wedge[\mathcal{Z}]\wedge(\pi^*\omega_\mathcal{B})^{\dim\mathcal{B}-1}=0.
\end{align*}
\end{proof}

\section{Stability of algebraic dynamical pairs}\label{section4}

\subsection{Several characterizations of stability}
{ Dynamical pairs (resp. stable dynamical pairs) were defined in the analytic setting in Definition~\ref{def:dynamical_pair} (resp. Definition~\ref{def_stable}). We now define the corresponding notions in the algebraic case.}
\begin{definition}\label{algebraic_pair}
A $p$-measurable dynamical pair $((\mathcal{X},\mathcal{f},\mathcal{L}),S)$ is \emph{algebraic} if
\begin{itemize}
\item[$(\dag_1)$] $(\mathcal{X},\mathcal{f},\mathcal{L})$ is an algebraic family of polarized endomorphisms.
\item[$(\dag_2)$]  $S=[\mathcal{Z}]$ where $\mathcal{Z}$ is an irreducible algebraic subvariety of $\mathcal{X}$ of codimension $p \leq k$ such that $\pi|_\mathcal{Z}:\mathcal{Z}\to \mathcal{B}$ is surjective,  where $k$ is the relative dimension of $\mathcal{X}\to\mathcal{B}$. 
\end{itemize}
 We say that the $p$-measurable dynamical pair $((\mathcal{X},\mathcal{f},\mathcal{L}),[\mathcal{Z}])$ is \emph{stable} if $(\mathcal{X}_\Lambda,\mathcal{f},\mathcal{L},[\mathcal{Z}])$ is stable for a given regular part $\Lambda$ of $(\mathcal{X},\mathcal{f},\mathcal{L})$.
\end{definition}
The notion of stability is independent of the chosen regular part $\Lambda$. Indeed, $(\mathcal{X}_\Lambda,\mathcal{f},\mathcal{L},[\mathcal{Z}])$ is stable if and only if the current $T_{f,\mathcal{Z}}$ vanishes identically on $\Lambda$. In particular, its trivial extension to $\mathcal{B}$ is identically zero.

\begin{example}\normalfont
 Let $(\mathcal{X},\mathcal{f},\mathcal{L})$ be an algebraic family of polarized endomorphisms.  Set 
\[\mathrm{Per}_\mathcal{f}(n,m):=\{z_0\in\mathcal{X}\, : \ \mathcal{f}^{n}(z_0)=\mathcal{f}^{m}(z_0)\}\]
for any $n>m\geq0$.  For $n-m$ large enough the set $\mathrm{Per}_\mathcal{f}(n,m)$ is non-empty by Theorem~\ref{prop:endopolarized} and defines a subvariety of $\mathcal{X}$ of pure dimension $\dim\mathcal{B}$, flat over some regular part $\Lambda$ (which depends on $n$ and $m$), by Lemma~\ref{lm:finite}.
As a consequence, the current $[\mathrm{Per}_\mathcal{f}(n,m)]$ is a closed positive horizontal $(k,k)$-current on $\mathcal{X}_\Lambda$.
As an immediate application of Theorem~\ref{tm:formulaheight} and Theorem~\ref{prop:endopolarized}, we have the following, see also Corollary~\ref{cor:stable-h0_intro}.
\end{example}

\begin{corollary}\label{cor:stable-h0}
For any $n>m\geq0$ with $n-m$ large enough, the set $\mathrm{Per}_\mathcal{f}(n,m)$ is non-empty. Moreover, for any irreducible component $\mathcal{C}$ of $\mathrm{Per}_\mathcal{f}(n,m)$, the $k$-measurable dynamical pair $((\mathcal{X},\mathcal{f},\mathcal{L}),[\mathcal{C}])$ is stable.
\end{corollary}

We use the notation of Theorem~\ref{tm:formulaheight}. Recall that, when $\mathcal{Z}$ is a subvariety of $\mathcal{X}$ of dimension $\ell+\dim\mathcal{B}$, the degree of $\mathcal{Z}$ relatively to the ample line bundle $\mathcal{M}=\mathcal{L}\otimes \pi^*(\mathcal{N})$ is given by
\[\deg_\mathcal{M}(\mathcal{Z})=\left(\mathcal{Z}\cdot c_1(\mathcal{M})^{\ell+\dim\mathcal{B}}\right)=\int_{\mathcal{Z}_\Lambda}\left(\widehat{\omega}+\check{\omega}_\mathcal{B}\right)^{\ell+\dim\mathcal{B}}=\|[\mathcal{Z}]\|,\]
where $\widehat{\omega}$ is cohomologous to $c_1(\mathcal{L})$

The following is a more general version of item 1.~of Theorem~\ref{tm:Pk_intro}. Observe that, in what follows, when $Z$ is a point (i.e. $p=k$), we have obviously $\deg(f|_Z)=1$, and item $3$ in Theorem~\ref{tm:caracterization-stable} below means $\deg_\mathcal{M}(\mathcal{Z}_n)=O(1)$, as claimed in the introduction. In particular, when $p=k$, the implication $2\implies 3$ in the following theorem follows immediately from Theorem~\ref{tm:formulaheight} and the fact that $\widehat{h}_f - h_{X,L}= O(1)$.

 The bifurcation current $T_{\mathcal{f},[\mathcal{Z}]}$ was defined in Section~\ref{sec:bifcur} and $\deg(f^n|_Z)$ in Section~\ref{multiplicity}.
\begin{theorem}\label{tm:caracterization-stable}
 Let $((\mathcal{X},\mathcal{f},\mathcal{L}),[\mathcal{Z}])$ be an algebraic $p$-measurable dynamical pair with $1\leq p\leq k$, which is a model of $(X,f,L,Z)$ over the field $\mathbf{K}$ of rational functions of a normal complex projective variety $\mathcal{B}$. Let $\mathcal{Z}_n$ be the Zariski closure of $\mathcal{f}^n(\mathcal{Z}_\Lambda)$ in $\mathcal{X}$.
Then the following are equivalent:
\begin{enumerate}
\item $((\mathcal{X},\mathcal{f},\mathcal{L}),[\mathcal{Z}])$ is stable,
\item $T_{\mathcal{f},[\mathcal{Z}]}=0$ as a closed positive $(1,1)$-current on $\Lambda$,
\item there exists $C>0$ such that for all $n\geq1$, 
\[C^{-1}\frac{d^{n(k-p)}}{\deg(f^n|_Z)}\leq \deg_\mathcal{M}(\mathcal{Z}_n)\leq C\frac{d^{n(k-p)}}{\deg(f^n|_Z)},\]
\item $\widehat{h}_{f}(Z)=0$.
\end{enumerate}
\end{theorem}

\begin{proof}
The equivalence between points 1 and 2 is the content of Proposition~\ref{lm:bifmeasure} and the equivalence between 2 and 4 is an immediate consequence of Theorem~\ref{tm:formulaheight}. We thus just need to prove the equivalence between 2 and 3.

\medskip

\noindent{\bf A formula for the degree $\deg_\mathcal{M}(\mathcal{Z}_n)$}. Let $\Lambda$ be a regular part over which  $\mathcal{Z}$ is flat. We  recall that, as algebraic cycles, for any $\lambda\in \Lambda$, we have
\[(f_\lambda^n)_*Z_\lambda=\deg(f_\lambda^n|_{Z_\lambda})\cdot f_\lambda^n(Z_\lambda)=\deg(f^n|_{Z})\cdot f_\lambda^n(Z_\lambda).\]
 We infer that, as currents on $\mathcal{X}_\Lambda$, we also have
\[(\mathcal{f}^n)_*[\mathcal{Z}]=\deg(f^n|_{Z})\cdot[\mathcal{f}^n(\mathcal{Z})].\]
As above, let $\ell:=k-p$ be the dimension of $Z$ and $\alpha_n:=d^{n\ell}/\deg(f^n|_Z)$. As $\check{\omega}_\mathcal{B}^{\dim \mathcal{B} +1}=0$,
\begin{align}
\deg_\mathcal{M}(\mathcal{Z}_n) & =\frac{1}{\deg(f^n|_Z)}\int_{\mathcal{X}_{\Lambda}}(\mathcal{f}^n)_*[\mathcal{Z}]\wedge \left(\widehat{\omega}+\check{\omega}_\mathcal{B}\right)^{\dim\mathcal{B}+\ell}\nonumber\\
& = \frac{1}{\deg(f^n|_Z)}\sum_{j=0}^{\dim\mathcal{B}}\binom{\dim\mathcal{B}+\ell}{ j+\ell}\int_{\mathcal{X}_{\Lambda}}(\mathcal{f}^n)_*[\mathcal{Z}]\wedge \widehat{\omega}^{j+\ell}\wedge\check{\omega}_\mathcal{B}^{\dim\mathcal{B}-j}\label{expand}
\end{align}

\noindent {\bf Proof of $3 \implies 2$}. 
For $j=1$ in the above sum, we have the term 
\[ \binom{\dim\mathcal{B}+\ell}{ 1+\ell}\frac{1}{\deg(f^n|_Z)}\int_{\mathcal{X}_{\Lambda}}(\mathcal{f}^n)_*[\mathcal{Z}]\wedge \widehat{\omega}^{\ell+1}\wedge\check{\omega}_\mathcal{B}^{\dim\mathcal{B}-1}. \]
The second point of Proposition~\ref{Stable=bounded} applied to $S=[\mathcal{Z}]$ and $q=\ell+1$ gives 
\[\int_{\mathcal{X}_{\Lambda}}(\mathcal{f}^n)_*[\mathcal{Z}]\wedge \widehat{\omega}^{\ell+1}\wedge\check{\omega}_\mathcal{B}^{\dim\mathcal{B}-1}- d^{n(\ell+1)} \int_{\mathcal{X}_{\Lambda}}\widehat{T}_\mathcal{f}^{\ell+1}\wedge[\mathcal{Z}]\wedge\check{\omega}_\mathcal{B}^{\dim\mathcal{B}-1}=O(d^{n\ell}).\]
Assume 2 does not hold. Then $T_{\mathcal{f},[\mathcal{Z}]}$ is non-zero, so
\[\deg_\mathcal{M}(\mathcal{Z}_n)\geq  d^n\alpha_n\int_{\mathcal{X}_{\Lambda}}\widehat{T}_\mathcal{f}^{\ell+1}\wedge[\mathcal{Z}]\wedge\check{\omega}_\mathcal{B}^{\dim\mathcal{B}-1}+O(\alpha_n),\]
so there is no $C>0$ such that $\deg_\mathcal{M}(\mathcal{f}^n(\mathcal{Z})) \leq C \alpha_n$. Hence 3 implies 2.

\medskip

\noindent {\bf Proof of $2 \implies 3$}. 
 Finally assume $T_{\mathcal{f},[\mathcal{Z}]} =0$ and we first want to prove $\deg_\mathcal{M}(\mathcal{Z}_n)=O(\alpha_n)$. Since $\mathcal{Z}$ is flat over a dense Zariski open set and since $\widehat{T}_f$ has continuous potentials, this implies $\widehat{T}_\mathcal{f}^{\ell+1}\wedge [\mathcal{Z}]=0$. 
 We first show $\int_{\mathcal{X}_{\Lambda}}(\mathcal{f}^n)_*[\mathcal{Z}]\wedge \widehat{\omega}^{j+\ell}\wedge\check{\omega}_\mathcal{B}^{\dim\mathcal{B}-j}=O(d^{n\ell})$ for all $1\leq j\leq \dim \mathcal{B}-1$. We use again the second point of Proposition~\ref{Stable=bounded} for $S=[\mathcal{Z}]$ and $q=\ell+j$ so  $\int_{\mathcal{X}_{\Lambda}}\widehat{T}_\mathcal{f}^{\ell+j}\wedge[\mathcal{Z}]\wedge\check{\omega}_\mathcal{B}^{\dim\mathcal{B}-j}=0$: 
\begin{align*}
	\int_{\mathcal{X}_{\Lambda}}(\mathcal{f}^n)_*[\mathcal{Z}]\wedge \widehat{\omega}^{j+\ell}\wedge\check{\omega}_\mathcal{B}^{\dim\mathcal{B}-j}&\leq C\sum\limits_{s<\ell + j} d^{ns}\|[\mathcal{Z}]\wedge\widehat{T}_\mathcal{f}^{s}\wedge  (\mathcal{f}^n)^*(\widehat{\omega})^{\ell + j-s-1}\| \\
	&\leq C\sum\limits_{s<\ell +1 } d^{ns}\|[\mathcal{Z}]\wedge\widehat{T}_\mathcal{f}^{s}\wedge  (\mathcal{f}^n)^*(\widehat{\omega})^{\ell + j-s-1}.\| 
	\end{align*}
 By induction on $j$, we show, for all $1\leq j\leq\dim\mathcal{B}$ and every $0\leq s \leq \ell +1$,
\[\|[\mathcal{Z}]\wedge (\mathcal{f}^n)^*(\widehat{\omega}^{\ell+j-s-1})\wedge\widehat{T}_\mathcal{f}^s\|=O(d^{n(\ell-s)}).\]
For $j=1$, it is immediate from the action of $f^*$ on cohomology. Pick $j$ and assume this holds for $j-1$ and all $s$, then we prove the estimate for $j$ by a descending induction on $s$. Indeed, it holds for $s=\ell +1$ since $\widehat{T}_\mathcal{f}^{\ell+1}\wedge [\mathcal{Z}]=0$. Assume it holds for $s+1 \leq \ell +1$. Point 1 of Proposition~\ref{Stable=bounded} then gives
\[\|[\mathcal{Z}]\wedge (\mathcal{f}^n)^*(\widehat{\omega}^{\ell+j-s-1})\wedge\widehat{T}_\mathcal{f}^s\|-d^n\|[\mathcal{Z}]\wedge (\mathcal{f}^n)^*(\widehat{\omega}^{\ell+j-s-2})\wedge\widehat{T}_\mathcal{f}^{s+1}\|=O(\| [\mathcal{Z}]\wedge (\mathcal{f}^n)^*(\widehat{\omega}^{\ell+j-s-2})\wedge\widehat{T}_\mathcal{f}^s\|), \]
so, by the induction hypothesis for $j-1$: 
\[\|[\mathcal{Z}]\wedge (\mathcal{f}^n)^*(\widehat{\omega}^{\ell+j-s-1})\wedge\widehat{T}_\mathcal{f}^s\|-d^n\|[\mathcal{Z}]\wedge (\mathcal{f}^n)^*(\widehat{\omega}^{\ell+j-s-2})\wedge\widehat{T}_\mathcal{f}^{s+1}\|=O(d^{n(\ell -s)}), \]
and by the induction hypothesis for $s+1$ (and $j$)
\[\|[\mathcal{Z}]\wedge (\mathcal{f}^n)^*(\widehat{\omega}^{\ell+j-s-1})\wedge\widehat{T}_\mathcal{f}^s\|= d^n O(d^{n(\ell -(s+1))}) +O(d^{n(\ell -s-1)}), \]
which gives the estimate. 

In particular, we have
\begin{align}\label{upperbound_1stpart}
	\int_{\mathcal{X}_{\Lambda}}(\mathcal{f}^n)_*[\mathcal{Z}]\wedge \widehat{\omega}^{j+\ell}\wedge\check{\omega}_\mathcal{B}^{\dim\mathcal{B}-j}&\leq C\sum\limits_{s<\ell +1 } d^{ns}O(d^{n(\ell -s)})=  O(d^{n(\ell-s)}).
\end{align}

Finally, for $j=0$, we have
\begin{align} \label{upperbound_2ndpart}
 \int_{\mathcal{X}_{\Lambda}}(f^n)_*[\mathcal{Z}]\wedge\widehat{\omega}^\ell\wedge \check{\omega}_\mathcal{B}^{\dim\mathcal{B}}
 &=\int_{\mathcal{X}_{\Lambda}}[\mathcal{Z}]\wedge(\mathcal{f}^n)^*(\widehat{\omega}^\ell)\wedge\check{\omega}_\mathcal{B}^{\dim\mathcal{B}} \nonumber \\
 & \leq C_5d^{n\ell}\|[\mathcal{Z}]\|,
\end{align}
for some $C_5>0$, since $\pi\circ \mathcal{f}^n=\pi$. 
Using both \eqref{upperbound_1stpart} and \eqref{upperbound_2ndpart} in \eqref{expand} gives the right hand side inequality in point 3 of the proposition.

 To get the other inequality, take some non empty open set $U \Subset \Lambda$, then we always have
\[\int_{\pi^{-1}(U)}\widehat{T}_\mathcal{f}^\ell\wedge[\mathcal{Z}]\wedge\check{\omega}_\mathcal{B}^{\dim\mathcal{B}}>0.\]
Indeed, for $\lambda\in U$, the slice along $X_\lambda$ of $\widehat{T}_\mathcal{f}^\ell\wedge[\mathcal{Z}]$ is $T_{f_\lambda}^\ell\wedge[Z_\lambda]$ which has mass $\deg(Z_\lambda)$ by Bézout, and 
\begin{align*}
	\int_{\pi^{-1}(U)}\widehat{T}_\mathcal{f}^\ell\wedge[\mathcal{Z}]\wedge\check{\omega}_\mathcal{B}^{\dim\mathcal{B}}=\int_U\pi_*\left(\widehat{T}_\mathcal{f}^\ell\wedge[\mathcal{Z}]\right)\cdot\omega_\mathcal{B}^{\dim\mathcal{B}},
\end{align*}
so the global mass of $\int_{\pi^{-1}(U)}\widehat{T}_\mathcal{f}^\ell\wedge[\mathcal{Z}]\wedge\check{\omega}_\mathcal{B}^{\dim\mathcal{B}}$ has to be positive.

Taking $j=0$ in \eqref{expand}, we have
\begin{align*}
\deg_\mathcal{M}(\mathcal{Z}_n)  \geq  \frac{1}{\deg(f^n|_Z)}\int_{\mathcal{X}_{\Lambda}}(\mathcal{f}^n)_*[\mathcal{Z}]\wedge \widehat{\omega}^{\ell}\wedge\check{\omega}_\mathcal{B}^{\dim\mathcal{B}}= \frac{1}{\deg(f^n|_Z)}\int_{\mathcal{X}_{\Lambda}}[\mathcal{Z}]\wedge (\mathcal{f}^n)^*(\widehat{\omega}^{\ell})\wedge\check{\omega}_\mathcal{B}^{\dim\mathcal{B}}. 
\end{align*}
By the convergence of $d^{-n}(\mathcal{f}^n)^*(\widehat{\omega})$ towards $\widehat{T}_\mathcal{f}$ with locally uniform convergence of the potentials, for $n$ large enough, we deduce, taking some non empty open set $U \Subset \Lambda$
\begin{align*}
	\deg_\mathcal{M}(\mathcal{Z}_n)  \geq  \frac{d^{n \ell}}{C_5\deg(f^n|_Z)}\int_{\pi^{-1}(U)}[\mathcal{Z}]\wedge \widehat{T}_\mathcal{f}^{\ell}\wedge\check{\omega}_\mathcal{B}^{\dim\mathcal{B}}. 
	\end{align*}
\end{proof}

\begin{remark}\normalfont
When $\dim Z>0$, the degree $\deg(f^n|_Z)$ could as well grow as $d^{\ell n}$ or remain constant equal to $1$. Consider, for example, a family of polarized endomorphisms of $\p^k$ given by polynomial endomorphisms of $\C^k$, then the hyperplane at infinity $\mathcal{H}_\infty$ defines an irreducible  algebraic subvariety of $\mathcal{X}$ of codimension $1$, it is fixed hence its height is zero and $\deg(\mathcal{f}^n|\mathcal{H}_\infty)= d^{n(k-1)}$.
More generally, we expect to have $\deg_\mathcal{M}(\mathcal{Z}_n)=O(1)$, when the pair $((\mathcal{X},\mathcal{f},\mathcal{L}),[\mathcal{Z}])$ is stable unless we have some isotriviality.  
\end{remark}

\subsection{A criterion for unstability: similarity}

The following is an adaptation of Tan Lei's similarity between the Julia set and the Mandelbrot set at a Misiurewicz parameter~\cite{Tan-similarity} and follows the idea of the proof of \cite[Theorem B]{AGMV}. Recall that a repelling periodic point $z_0$ of  $f_{\lambda_0}$ can be followed analytically as a repelling periodic point in a neighborhood of $\lambda_0$ by the Implicit Function Theorem.

\begin{lemma}\label{lm:preperstable}  Let $((\mathcal{X},\mathcal{f},\mathcal{L}),[\mathcal{Z}])$ be an algebraic $k$-measurable dynamical pair with regular part $\Lambda$. 
 Pick $\lambda_0\in\Lambda$. Assume $z_0\in (X_{\lambda_0})_{\mathrm{reg}}\cap \mathcal{Z}$ is a repelling periodic point for $f_{\lambda_0}$. Let $z:\lambda\in(\Lambda,\lambda_0)\mapsto z(\lambda)\in(\mathcal{X},z_0)$ be the local analytic continuation of $z_0$ as a repelling periodic point of $f_\lambda$. Then we have the alternative
	\begin{enumerate}
		\item either $\lambda_0\in \mathrm{supp}(T_{\mathcal{f},[\mathcal{Z}]})$, in particular $((\mathcal{X},\mathcal{f},\mathcal{L}),[\mathcal{Z}])$ is not stable,
		\item or there is a neighborhood $U$ of $\lambda_0$ in $\Lambda$ such that the graph $\Gamma_z$ satisfies $\Gamma_z\subset \mathcal{Z}\cap \pi^{-1}(U)$.
	\end{enumerate}
\end{lemma}

\begin{proof}
	Observe that if point $2$ occurs, then, as $\mathcal{Z}$ is irreducible and intersect persistently some $\mathrm{Per}_\mathcal{f}(n,0)$, it has to be globally contained in $\mathrm{Per}_\mathcal{f}(n,0)$ so it is stable by Corollary~\ref{cor:stable-h0}.
	
	Up to replacing $\mathcal{f}$ by an iterate, we can assume that $z_0$ is fixed. We assume the graph $\Gamma_z$ of the map $z$ satisfies $\Gamma_z\not\subset\mathcal{Z}$. First, as $T_{\mathcal{f},[\mathcal{Z}]}$ is a closed positive current on $\Lambda$ with continuous potentials, we can reduce to the case where $\dim\mathcal{B}=1$. Indeed, for any curve $\mathcal{C}\subset\mathcal{B}$, $\supp\left(T_{\mathcal{f},[\mathcal{Z}]}\wedge[\mathcal{C}]\right) \subset \supp(T_{\mathcal{f},[\mathcal{Z}]})$, see, e.g., \cite[Lemma~6.3]{Article1}.
	
	\medskip
	
	We thus assume that $\dim\mathcal{B}=1$, and that there exists an open neighborhood $U$ of $\lambda_0$ in $\mathcal{B}$ such that the intersection $\mathcal{Z}\cap \Gamma_z\cap\pi^{-1}(U) =\{z_0\}$, up to reducing $U$. We want to prove that $T_{\mathcal{f},[\mathcal{Z}]}\neq0$ on any neighborhood of $\lambda_0$.
	As $z(\lambda)$ is repelling for all $\lambda\in U$, up to reducing $U$, we can assume there exist $K>1$ and $\delta>0$ such that 
	\[d_{X_\lambda}(f_\lambda(x),f_\lambda(z(\lambda))) \geq Kd_{X_\lambda}(x,z(\lambda)),\]
	for all $\lambda\in  U$ and all $x\in \B_{X_\lambda}(z(\lambda),\delta)\subset X_\lambda$.
	In particular, there exists a neighborhood $\Omega \Subset \B(z_0,\varepsilon)$ of $z_0$ in $\mathcal{X}$ such that the analytic map $\mathcal{f}:\Omega\to \mathcal{f}(\Omega)$ is proper with $\Omega\subset \mathcal{f}(\Omega)$ and $\pi(\Omega)\subset U$. In addition, if $\lambda\in\pi(\Omega)$, we have $\Omega\cap X_\lambda\Subset \mathcal{f}(\Omega\cap X_\lambda)$.
	Then
	\begin{align*}
	d^{n}\int_{\Omega}\widehat{T}_\mathcal{f}\wedge [\mathcal{Z}]& =\int_{\Omega}(\mathcal{f}^{n})^*\widehat{T}_\mathcal{f}\wedge [\mathcal{Z}]=\int(\mathcal{f}^{n})^*\widehat{T}_\mathcal{f}\wedge \left(\mathbf{1}_{\Omega}\cdot [\mathcal{Z}]\right)\\
	& = \int\widehat{T}_\mathcal{f}\wedge (\mathcal{f}^{n})_*\left(\mathbf{1}_{\Omega}\cdot [\mathcal{Z}]\right)\geq\int\widehat{T}_\mathcal{f}\wedge \mathbf{1}_{\Omega}\cdot (\mathcal{f}^{n})_*\left([\mathcal{Z}]\right).
	\end{align*}
Let $S_n$ be the current of integration on the connected component of $ \mathcal{f}^{n}\left(\Omega\cap \mathcal{Z}\right)\cap \Omega $ containing $z_0$.
If we assume, by contradiction, that $((\mathcal{X},\mathcal{f},\mathcal{L}),[\mathcal{Z}])$ is a stable analytic pair over $\pi(\Omega)$, by Proposition~\ref{lm:bifmeasure}, the sequence $S_n$ defines a sequence of integration currents on analytic subsets of $\Omega$ containing $z_0$ and of uniformly bounded mass. By a theorem of Bishop, any weak limit of $(S_n)$ is the integration current on an analytic curve of $\Omega$ containing $z_0$. Let $Z_0$ be such a weak limit. By construction, it is supported in $X_{\lambda_0}\cap \Omega$ whence
\[\int_{\Omega}\widehat{T}_\mathcal{f}\wedge [Z_0]=\int_{X_{\lambda_0}\cap \Omega}T_{f_{\lambda_0}}\wedge [Z_0].\]
Since $z_0$ is repelling, the sequence of iterates $(f_{\lambda_0}^n|_{Z_0})_n$  cannot be equicontinuous, so $z_0$ belongs to the support of the current $T_{f_{\lambda_0}}\wedge [Z_0]$ (see the proof of Theorem~1.6.5 of \cite{Sibony} for details), hence
\[\int_{X_{\lambda_0}\cap \Omega}T_{f_{\lambda_0}}\wedge [Z_0]>0.\]
 This is is a contradiction since 
 \[ d^{n}\int_{\Omega}\widehat{T}_\mathcal{f}\wedge [\mathcal{Z}]\geq\int\widehat{T}_\mathcal{f}\wedge \mathbf{1}_{\Omega}\cdot (\mathcal{f}^{n})_*\left([\mathcal{Z}]\right)>0\]
 for $n$ large enough.  The proof is complete.
\end{proof}

\subsection{The geometric dynamical Northcott property}

 \begin{theorem}\label{tm:Pk}
 Let $(\mathcal{X},\mathcal{f},\mathcal{L})$ be a non-isotrivial algebraic family of polarized endomorphisms over a normal complex projective variety $\mathcal{B}$. Fix an integer $D\geq1$ and let $\mathcal{N}$ be a very ample line bundle on $\mathcal{B}$ such that $\mathcal{M}:=\mathcal{L}\otimes \pi^*(\mathcal{N})$ is ample on $\mathcal{X}$. 
There exist a possibly empty subvariety $\mathcal{Y}_D$ of $\mathcal{X}$ such that $\pi:\mathcal{Y}_D\to\mathcal{B}$ is surjective when $\mathcal{Y}_D\neq\varnothing$ and $N\geq0$ such that
\begin{enumerate}
\item $f(\mathcal{Y}_D)=\mathcal{Y}_D$, and for any periodic irreducible component $\mathcal{V}$ of $\mathcal{Y}_D$, of period $n$, the family of polarized endomorphisms $(\mathcal{V},\mathcal{f}^{n}|_\mathcal{V},\mathcal{L}|_\mathcal{V})$ is isotrivial,
\item for any rational section $a: \mathcal{B} \dasharrow \mathcal{X}$ and any $n$, let $\mathcal{C}_n(a)$ be the Zariski closure in $\mathcal{X}$ of the image of $\mathcal{f}^n(a)$. If we assume that $\deg_\mathcal{M}(\mathcal{C}_n(a))\leq D$ for all $n\geq1$, then $\mathcal{Y}_D$ contains $\mathcal{C}_N(a)$.
\end{enumerate}
 \end{theorem} 

 \begin{remark}\normalfont
\begin{enumerate}
 \item  In the case where any irreducible component of $\mathcal{Y}_D$ has dimension $\dim\mathcal{B}$, the variety $\mathcal{Y}_D$ is pointwise periodic, i.e.\  there exists $n>0$ such that $\mathcal{f}^n(x)=x$, for all $x\in\mathcal{Y}_D$.
\item  In fact we will see in Theorem~\ref{tm:mainPk} that we can make $\mathcal{Y}_D$ \emph{independent} of $D$.
\item While $\mathcal{Y}_D$ can be empty for small $D$, Corollary~\ref{cor:stable-h0} guarantees that, up to a base change, it is not empty for $D$ large enough. 
\end{enumerate}
\end{remark}
 
 \begin{proof}
 	First, observe that it is enough to prove the theorem in the case where $\mathcal{X}$ is normal which we assume from now on.  
 	
 		Recall that the Chow variety $\mathrm{Ch}_{\dim\mathcal{B},D}(\mathcal{X})$ is the (non reduced) variety of algebraic cycles of $\mathcal{X}$ of dimension $\dim\mathcal{B}$ and degree at most $D$, relatively to $\mathcal{M}$ (see, e.g. \cite[Chapter~I]{Kollar}).
 Let $\mathcal{D}$ denote the set of subvarieties of $\mathcal{X}$ of the form $\mathcal{C}_0(a)$ for some rational section $a: \mathcal{B} \dasharrow \mathcal{X} $. For any $n\in \N$, let 
 \[\mathcal{Z}_{n,D}:= \{\mathcal{C} \in  \mathcal{D}\cap \mathrm{Ch}_{\dim\mathcal{B},D}(\mathcal{X}), \ \forall j \leq n, \ \deg_\mathcal{M}(\mathcal{f}^j(\mathcal{C})) \leq D \}\]
 	and $\overline{\mathcal{Z}_{n,D}}$ its Zariski closure in $\mathrm{Ch}_{\dim\mathcal{B},D}(\mathcal{X})$.  Then, the decreasing intersection
 	\[\bigcap_{n = 0}^{N}  \overline{\mathcal{Z}_{n,D}} \]
 	is eventually constant by noetherianity so there exists an $n_0$ such that 
 	\[ \mathcal{Z}_D^1:=\bigcap_{n = 0}^{n_0}  \overline{\mathcal{Z}_{n,D}} =\bigcap_{n \geq 0}  \overline{\mathcal{Z}_{n,D}}=\overline{\mathcal{Z}_{n_0,D}},\]
 	is a projective variety. 	Let $W$ be an irreducible component of $\mathcal{Z}_D^1$. Then $f$ induces a rational map $W\dashrightarrow \mathrm{Ch}_{\dim\mathcal{B},dD}(\mathcal{X})$ whose image is contained in $\mathcal{Z}_D^1\subset \mathrm{Ch}_{\dim\mathcal{B},D}(\mathcal{X})$ by Zariski density of $\mathcal{Z}_{n_0,D}$ in $W$.
 Denote by $g: \mathcal{Z}_D^1 \dasharrow \mathcal{Z}_D^1$ the induced rational map. 
 By construction, for each irreducible component $W$ of $\mathcal{Z}_{n_0,D}$, $g$ is well defined - as a rational map - on a dense Zariski open set $U$ of $W$ and satisfies $g(U) \subset \mathcal{Z}_{n_0,D}$, whence for any irreducible component $V$ of $\mathcal{Z}_D^1$, the iterates of $g$ are all well-defined as rational maps on a Zariski open subset of $V$.

 In particular, under iteration of $g$, any irreducible component of $\mathcal{Z}_D^1$  is sent into a periodic irreducible component and its image is not fully contained in the indeterminacy set of $g$, since the image of $g$ is contained in $\mathcal{Z}_{n_0,D}$. Let $\mathcal{Z}_D$ be such a periodic irreducible component of $\mathcal{Z}_D^1$. Note that all the $\mathcal{V}$'s lying in a given irreducible component have to be cohomologous. As we have proved in Lemma~\ref{lm:isotriviality-iterate} that $(\mathcal{X},\mathcal{f},\mathcal{L})$ is isotrivial  if and only if $(\mathcal{X},\mathcal{f}^n,\mathcal{L})$ is for some $n\geq1$, we may assume $g(\mathcal{Z}_D)\subset\mathcal{Z}_D$. As $\mathcal{f}$ has finite fibers over any regular part $\Lambda$, $g(\mathcal{Z}_D)$ has the same dimension as $\mathcal{Z}_D$. If not, a general element $\mathcal{W}$ in  $g(\mathcal{Z}_D)$ would have infinitely many preimages $\mathcal{W}_i$ in $\mathcal{Z}_D$ and, for a general parameter $\lambda \in \Lambda$,  $\mathcal{W}_i \cap  X_\lambda \subset f_\lambda^{-1}(\mathcal{W}\cap X_\lambda)$ (which is finite since $\mathcal{W}\cap X_\lambda$ is finite), a contradiction.  As $\mathcal{Z}_D$ is irreducible, $g:\mathcal{Z}_D \dashrightarrow \mathcal{Z}_D$ is dominant.

We now define 
 \[\widehat{\mathcal{Z}}_D:=\{(\mathcal{V},x)\in \mathcal{Z}_D\times \mathcal{X}\, : \ x\in \mathcal{V}\}.\]
Assume that the canonical projection
$\Pi:(\mathcal{V},x)\longmapsto x$
 is dominant, whence $\Pi$ is surjective. We want to prove $(\mathcal{X},\mathcal{f},\mathcal{L})$ is isotrivial.
 For a given regular part $\Lambda$, we set $\widehat{\mathcal{Z}}_D^\Lambda:=\Pi^{-1}(\mathcal{X}_\Lambda)$, we have the following lemma.

\begin{lemma}\label{lm:isom}
There exists a regular part ${\Lambda}$ such that 
\begin{enumerate}
\item the map $\Pi|_{\widehat{\mathcal{Z}}_D^{\Lambda}}:\widehat{\mathcal{Z}}_D^{\Lambda}\to\mathcal{X}_{{\Lambda}}$ is an isomorphism,
\item $\dim(\mathcal{Z}_D)= k$,
\item if
$p:\widehat{\mathcal{Z}}_D^{\Lambda}\to\mathcal{Z}_D$ is the projection given by $p(\mathcal{V},z)=\mathcal{V}$ then the map 
\[\Psi:=p\circ \left(\Pi|_{\widehat{\mathcal{Z}}_D^{\Lambda}}\right)^{-1}:\mathcal{X}_{\Lambda}\longrightarrow \mathcal{Z}_D\]
is regular and its fibers are of the form $\mathcal{V}_{\Lambda}$ for some $\mathcal{V}\in\mathcal{Z}_D$,
\item $\Psi \circ \mathcal{f}=g \circ \Psi$.
\end{enumerate}
\end{lemma} 

We take the lemma for granted.  Up to normalizing $\mathcal{Z}_D$, we may assume it is normal.
Note that, by assumption,  there is $\mathcal{C}_0\in\mathcal{Z}_D$ with $\left(\mathcal{C}_0\cdot X_\lambda\right)=1$,  for all $\lambda$ in $ \Lambda$. As all fibers of $\pi$ over $\Lambda$ (resp. all varieties in $\mathcal{Z}_D$) are cohomologous and as the intersection can be computed in cohomology, we deduce that
\[\left(\mathcal{V}\cdot X_\lambda\right)=1\]
for all $\lambda\in\Lambda$ and all $\mathcal{V}\in\mathcal{Z}_D$. Let now $\psi_\lambda:=\Psi|_{X_\lambda}:X_\lambda\rightarrow \mathcal{Z}_D$. $\psi_\lambda$ is a morphism by construction. Finally, since the topological degree of $\psi_\lambda$ is exactly $\left(\mathcal{V}\cdot X_\lambda\right)$ for a general variety $\mathcal{V}\in\mathcal{Z}_D$, we get that $\psi_\lambda$ is an isomorphism which conjugates $f_\lambda$ to $g$. In particular, $g$ is an endomorphism.

\smallskip

We have proved that for, any $\lambda\in\Lambda$, the morphism $f_\lambda$ is conjugated by an isomorphism $\psi_\lambda:X_\lambda\to \mathcal{Z}_D$ to the endomorphism $g$. In particular, given a fixed $\lambda_0\in \Lambda$, the map $\Phi:=(\psi_{\lambda_0}^{-1}\circ \Psi_\Lambda,\pi):\mathcal{X}_\Lambda\to X_{\lambda_0}\times \Lambda$ is an isomorphism which satisfies $\Phi\circ f=(f_{\lambda_0},\pi)\circ \Phi$. In particular, if $\phi_\lambda:X_\lambda\to X_{\lambda_0}$ is the induced isomorphism, $(\phi_\lambda^{-1})^*L_\lambda$ depends continuously on $\lambda$, whence its class in the Picard group $\mathrm{Pic}(X_{\lambda_0})$ also. By \cite[Lemma~2.1 $\&$ Lemma~2.3 (i)]{Nakayama-Zhang}, the subset of ample line bundles in the Picard group of a complex projective variety which can polarize a given morphism is discrete and $\Lambda$ is connected. Whence $(\phi_\lambda^{-1})^*L_\lambda$ must be constant. Since $\phi_{\lambda_0}=\mathrm{id}_{X_{\lambda_0}}$, we infer $\phi_\lambda^*L_{\lambda_0}=L_\lambda$ for any $\lambda\in \Lambda$.
 We thus have proved $(\mathcal{X},\mathcal{f},\mathcal{L})$ is isotrivial.

\medskip

As a consequence, when $(\mathcal{X},\mathcal{f},\mathcal{L})$ is non-isotrivial, the image $\mathcal{Y}$ of $\Pi_{\widehat{\mathcal{Z}}_D}:\widehat{\mathcal{Z}}_D\to\mathcal{X}$ is a strict subvariety of $\mathcal{X}$ which is invariant by $\mathcal{f}$. Moreover, as $\mathcal{f}$ has finite fibers, we have $\mathcal{f}(\mathcal{Y})=\mathcal{Y}$. Let $\mathsf{n}:\widehat{\mathcal{Y}}\to\mathcal{Y}$  be the normalization of $\mathcal{Y}$. If $(\mathcal{Y},\mathcal{f}|_{\mathcal{Y}},\mathcal{L}|_{\mathcal{Y}})$ is non-isotrivial, we  end up with a contradiction applying the same strategy as above to $(\widehat{\mathcal{Y}},\widehat{\mathcal{f}}|_{\mathcal{Y}},\mathsf{n}^*\mathcal{L}|_{\mathcal{Y}})$. Thus $(\mathcal{Y},\mathcal{f}|_{\mathcal{Y}},\mathcal{L}|_{\mathcal{Y}})$ is isotrivial.

Finally, if $\mathcal{W}_1,\ldots,\mathcal{W}_Q$ are all periodic irreducible components of $\mathcal{Z}_D^1$, we let $\mathcal{Y}_D$ be the union of the images of $\widehat{\mathcal{W}}_i$ under the natural projections $\Pi_{\widehat{\mathcal{W}}_i}:\widehat{\mathcal{W}}_i\to\mathcal{X}$ defined as above. By construction of the varieties $\mathcal{W}_i$, there exists an integer $N\geq1$ such that any irreducible subvariety $\mathcal{C}\in \mathcal{D}$, with $\sup_n\deg_\mathcal{M}(\mathcal{f}^n(\mathcal{C}))\leq D$, satisfies $\mathcal{f}^N(\mathcal{C})\subset\mathcal{Y}_D$. This concludes the proof.
\end{proof}

 \begin{proof}[Proof of Lemma~\ref{lm:isom}]
\noindent\textbf{Step 1.} Take $\Lambda$ to be the maximal regular part. Let $S$ be the set of points $x\in\mathcal{X}_\Lambda$ such that  $x\in (X_{\lambda})_{\mathrm{reg}}$ and $x$ is a repelling periodic point of $f_{\lambda}$, where $\lambda:=\pi(x)$. Pick  $\lambda_0\in\Lambda$. Since repelling periodic points of $f_{\lambda_0}$ are Zariski dense in $X_{\lambda_0}$ by Theorem~\ref{prop:endopolarized}, we can pick $z_0\in S$ with $\pi(z_0)=\lambda_0$. Let $p\geq1$ be its exact period. By the Implicit Function Theorem, there exist a neighborhood $U\subset\Lambda$ of $\lambda_0$ and a local section $z:U\to\mathcal{X}$ of $\pi:\mathcal{X}\to B$ such that
 \[f_\lambda^p(z(\lambda))=z(\lambda), \ \ \lambda\in U,\]
 and $z(\lambda)\in (X_\lambda)_{\mathrm{reg}}$ is repelling for $f_\lambda$ for all $\lambda\in U$. By assumption, there exists $\mathcal{V}\in\mathcal{Z}_D$ such that $(z_0,\lambda_0)\in \mathcal{V}$. 
 Since $\mathcal{V}\in \mathcal{Z}_D$, one has $\deg_{\mathcal{M}}(\mathcal{f}^n(\mathcal{V}))\leq D$ for all $n\geq0$. By Theorem~\ref{tm:caracterization-stable}, this implies that $((\mathcal{X},\mathcal{f},\mathcal{L}),[\mathcal{V}])$ is stable. In particular, Lemma \ref{lm:preperstable} implies that $\mathcal{V}\cap \pi^{-1}(U)= z(U)$. This implies that $\mathcal{V}$ coincides with the graph of $z$, whence $\Pi^{-1}\{z_0\}$ is a singleton.
 
By Theorem~\ref{prop:endopolarized}, $S$ is a Zariski dense subset of $\mathcal{X}_\Lambda$ and we proved that $\Pi^{-1}\{x\}$ is a singleton for all $x\in S$.
Whence there exist dense Zariski open sets $\mathcal{U}\subset\widehat{\mathcal{Z}}_D^\Lambda$ and $\mathcal{W}\subset \mathcal{X}_\Lambda$ such that $\Pi|_{\mathcal{U}}:\mathcal{U}\longrightarrow \mathcal{W}$ is a biholomorphism.
Up to replacing $\Lambda$ with $\Lambda\setminus G$ where $G$ is a subvariety, we can assume $\mathcal{W}=\pi^{-1}(\Lambda)\setminus (\mathcal{X}'\cup\mathcal{X}'')$, where $\mathcal{X}'$ and $\mathcal{X}''$ are strict subvarieties of $\mathcal{X}$, flat over $\Lambda$, which are the respective Zariski closures of $\{x\in \mathcal{X}_\Lambda\, ; \ 1<\mathrm{Card}(\Pi^{-1}\{x\})<+\infty\}$ and $\{x\in \mathcal{X}_\Lambda\, ; \ \Pi^{-1}\{x\}$ is infinite$\}$.

\medskip

\par\noindent \textbf{Step 2.} Note that $\mathcal{f}^{-1}(\mathcal{X}_\Lambda'')=\mathcal{X}_\Lambda''=\mathcal{f}(\mathcal{X}_\Lambda'')$. Indeed, $\mathcal{f}$ is a finite morphism on $\mathcal{X}_\Lambda$ thus  if $x\in \mathcal{X}_\Lambda\setminus \mathcal{X}_\Lambda''$, then any $\mathcal{V}\in\mathcal{Z}_D$ passing through $x$ has a unique image by $\mathcal{f}$ and at most finitely many preimages, whence 
\[\mathcal{f}^{-1}(\mathcal{X}_\Lambda\setminus \mathcal{X}_\Lambda'')=\mathcal{X}_\Lambda\setminus \mathcal{X}_\Lambda''=\mathcal{f}(\mathcal{X}_\Lambda\setminus \mathcal{X}_\Lambda''),\]
and thus $\mathcal{f}^{-1}(\mathcal{X}_\Lambda'')=\mathcal{X}_\Lambda''=\mathcal{f}(\mathcal{X}_\Lambda'')$. We thus have proved so far that the map $\Pi|_{\widehat{\mathcal{Z}}_D^\Lambda}:\widehat{\mathcal{Z}}_D^\Lambda\to\mathcal{X}_\Lambda$ restricts to $\mathcal{Z}^{(0)}:=\Pi|_{\widehat{\mathcal{Z}}_D^\Lambda}^{-1}(\mathcal{X}_\Lambda\setminus \mathcal{X}_\Lambda'')$ as a finite birational morphism. Since $\mathcal{X}_\Lambda$ is normal, this implies $\Pi|_{\widehat{\mathcal{Z}}_D^\Lambda}|_{\mathcal{Z}^{(0)}}$ is an isomorphism from $\mathcal{Z}^{(0)}$ to $\mathcal{X}_\Lambda\setminus \mathcal{X}_\Lambda''$. In particular, $\mathcal{X}_\Lambda'=\varnothing$.

\medskip

\par\noindent \textbf{Step 3.} We now prove that $\mathcal{X}_\Lambda''=\varnothing$. 
Pick a general parameter $\lambda_0\in \Lambda$. Assume that for a general parameter $\lambda\in \Lambda$, there is $x\in X_\lambda\cap\mathcal{X}_\Lambda''$ such that there is $\mathcal{V}\in \mathcal{Z}_D$ with $x\in \mathcal{V}$ and $\mathcal{V}_\Lambda\not\subset\mathcal{X}_\Lambda''$ and let us reach a contradiction. Let
\[W_\lambda:=\bigcup_{\mathcal{V}} X_{\lambda_0}\cap \mathcal{V},\]
where the union ranges over $\{\mathcal{V}\in \mathcal{Z}_D; \ \mathcal{V}\cap X_\lambda \subset \mathcal{X}_\Lambda''\}$ which is Zariski closed. By construction, the set $W_\lambda$ is a Zariski closed totally invariant subset for $f_{\lambda_0}$. By hypothesis, it is non-empty. Therefore, to a general $\lambda\neq\lambda_0$, we can associate a non-empty closed subset $W_\lambda$. Since there is one and only one $\mathcal{V}\in\mathcal{Z}_D$ passing through a point in $\mathcal{W}$, for any general $\lambda\neq\lambda'$, we have $W_\lambda\cap \mathcal{W}\neq W_{\lambda'}\cap \mathcal{W}$, whence $f_{\lambda_0}$ has infinitely many distinct totally invariant Zariski closed subsets. This is a contradiction by~\cite[Theorem~1.47]{dinhsibony2}. 
Assume now that $\dim\mathcal{X}_\Lambda''>\dim \mathcal{B}$, we cap apply step 1 to the family $(\mathcal{X}_\Lambda'',\mathcal{f}|_{\mathcal{X}_\Lambda''},\mathcal{L}|_{\mathcal{X}_\Lambda''})$.
So the set $\mathcal{X}_\Lambda''$ is at worst a  subvariety of dimension $\dim\mathcal{B}$ which is flat over $\Lambda$ and totally invariant by $f$. But it is supposed to contain infinitely many distinct such subvarieties of $\mathcal{Z}_D$ which pass by points in $\mathcal{X}_\Lambda''$. This is impossible, whence $\mathcal{X}_\Lambda''=\varnothing$.
 \end{proof}

\subsection{Proof of the second part of Theorem~\ref{tm:Pk_intro}}
Now we can deduce the second part Theorem~\ref{tm:Pk_intro} from Theorem~\ref{tm:Pk}. Let us rephrase the second part of Theorem~\ref{tm:Pk_intro}. 

\begin{theorem}\label{tm:mainPk}
Let $(\mathcal{X},\mathcal{f},\mathcal{L})$ be a non-isotrivial family of polarized endomorphisms.
There exist integers $N\geq0$ and $D_0\geq1$ and a proper subvariety $\mathcal{Y}$ of $\mathcal{X}$ such that $\pi$ is surjective on every irreducible component of $\mathcal{Y}$, whose irreducible components are periodic, and such that
\begin{enumerate}
\item $\mathcal{f}(\mathcal{Y})=\mathcal{Y}$ and for any periodic irreducible component $\mathcal{V}$ of $\mathcal{Y}$ of period $n$, the family of polarized endomorphisms $(\mathcal{V},\mathcal{f}^{n}|_\mathcal{V},\mathcal{L}|_\mathcal{V})$ is isotrivial,
\item If $a:\mathcal{B}\dashrightarrow\mathcal{X}$ is a marked point such that $((\mathcal{X},\mathcal{f},\mathcal{L}),a)$ is stable, then
\begin{enumerate}
\item $\lambda\mapsto f_\lambda^N(a(\lambda))$ is a section of $\pi|_{\mathcal{Y}}:\mathcal{Y}\to\mathcal{B}$,
\item if $\mathcal{C}_0(a)$ is the Zariski closure of the image of a regular part over which $a$ is regular, we have $\deg_\mathcal{M}(\mathcal{C}_0(a))\leq D_0$.
\end{enumerate}
\end{enumerate}
\end{theorem}

\begin{proof}
For a section $a$ and an integer $n\geq0$, denote by $a_n$ the section given by $\lambda\mapsto f^n_\lambda(a(\lambda))$. Assume there is a rational section $a$ of $\pi$ with $\deg_\mathcal{M}(\mathcal{C}_n(a))\leq D$ for all $n\geq0$, where $\mathcal{C}_n(a)$ is the Zariski closure of $a_n(\Lambda)$ for some regular part over which $a$ is regular. By Theorem~\ref{tm:Pk}, there exist a subvariety $\mathcal{Y}_D$ of $\mathcal{X}$ such that the restriction of $\pi$ to any irreducible components of $\mathcal{Y}_D$ is surjective and an integer $N_D\geq1$ such that properties $1$ and $2$ are satisfied.

Moreover, if $D\leq D'$, then $\mathcal{Y}_D\subset \mathcal{Y}_{D'}$. 
	All there is left to prove is the existence of $D_0\geq1$ such that if $\deg_\mathcal{M}(\mathcal{C}_0(a))>D_0$, then $((\mathcal{X},\mathcal{f},\mathcal{L}),a)$ is not stable.

Let $X$ be the generic fiber of $\pi$, $L:=\mathcal{L}|_X$ and $f:=\mathcal{f}|_X$, so that $(\mathcal{X},\mathcal{f},\mathcal{L})$ is a model for $(X,f,L)$ with regular part $\Lambda$. Denote by $\widehat{h}_f$ the canonical height function of $(X,f,L)$ as on Section~\ref{sec:height}. Any rational section corresponds to some $a\in X_\mathbf{K}$, where $\mathbf{K}=\C(\mathcal{B})$ and
\begin{align*}
h_{X,L}(a) & =\left(\mathcal{C}_0(a)\cdot c_1(\mathcal{L})\cdot c_1(\pi^*\mathcal{N})^{\dim\mathcal{B}-1}\right)\\
& =\deg_\mathcal{M}(\mathcal{C}_0(a))-\left(\mathcal{C}_0(a)\cdot c_1(\pi^*\mathcal{N})^{\dim\mathcal{B}}\right)\\
& =\deg_\mathcal{M}(\mathcal{C}_0(a))-1,
\end{align*}
since $\mathcal{C}_0(a)$ is the Zariski closure of the image of a rational section of $\pi$. Indeed, the size of the field extension $\mathbf{K}(a)$ of $\mathbf{K}$ is exactly the intersection number $(\mathcal{C}_0(a)\cdot c_1(\pi^*\mathcal{N})^{\dim\mathcal{B}})$.
Recall from Section~\ref{sec:height} that $h_{X,L}=\widehat{h}_f+O(1)$, so there exists a constant $C>0$ such that $\widehat{h}_f\geq h_{X,L}-C$, whence
\[\widehat{h}_f(a)\geq \deg_\mathcal{M}(\mathcal{C}_0(a))-(C+1).\]
If $\deg_\mathcal{M}(\mathcal{C}_0(a))> C +1$, we thus find $\widehat{h}_f(a)>0$. By Theorem~\ref{tm:caracterization-stable}, we deduce that $((\mathcal{X},\mathcal{f},\mathcal{L}),a)$ is not stable, ending the proof.
\end{proof}

We now deduce Corollary~\ref{cor:Northcott_function_field} from Theorem~\ref{tm:mainPk}, following the argument of \cite[\S~3.1]{Cantat-Gao-Habegger-Xie}.

\begin{proof}[Proof of Corollary~\ref{cor:Northcott_function_field}]
Here the polarized endomorphism $(X,f,L)$ is defined over a function field $\mathbf{K}:=\mathbf{k}(\mathcal{B})$ where $\mathbf{k}$ is a field of characteristic zero, where $\mathcal{B}$ is a normal projective $\mathbf{k}$-variety. There is an algebraically closed subfield $\mathbf{k}_0$ of $\mathbf{k}$ of finite transcendence degree over $\mathbb{Q}$ such that, via base change, we can assume $\mathcal{B}$ comes from a variety defined over $\mathbf{k}_0$. We can also assume $(X,f,L)$ comes from a polarized endomorphism defined over $\mathbf{k}_0(\mathcal{B})$. Now, the field $\mathbf{k}_0$ can be embedded into $\mathbb{C}$, so that $\mathcal{B}$ induces a complex variety $\mathcal{B}_\mathbb{C}$ and $(X,f,L)$ induces a polarized endomorphism $(X_\mathbb{C},f_\mathbb{C},L_\mathbb{C})$ defined over $\mathbf{K}':=\mathbb{C}(\mathcal{B}_\mathbb{C})$.

We now need to justify that if $\widehat{h}_f(z)=0$ and if $z_\mathbb{C}\in X_\mathbb{C}(\mathbf{K}')$ is induced by $z$, then $\widehat{h}_{f_\mathbb{C}}(z_\mathbb{C})=0$. Let $\mathcal{N}$ be an ample line bundle on $\mathcal{B}$ and $\mathcal{N}_\mathbb{C}$ be the induced ample line bundle on $\mathcal{B}_\mathbb{C}$. Let $(\mathcal{X},\mathcal{f},\mathcal{L})$ be a model of $(X,f,L)$ and let $(\mathcal{X}_\mathbb{C},\mathcal{f}_\mathbb{C},\mathcal{L}_\mathbb{C})$ be a model of $(X_\mathbb{C},f_\mathbb{C},L_\mathbb{C})$. Let finally $\mathcal{z}$ (resp. $\mathcal{z}_\mathbb{C}$) be the rational section of $\pi:\mathcal{X}\to\mathcal{B}$ induced by $z$ (resp. of $\pi_\mathbb{C}:\mathcal{X}_\mathbb{C}\to\mathcal{B}_\mathbb{C}$ induced by $z_\mathbb{C}$). Denote as above by $\mathcal{C}_n(z)$ (resp. $\mathcal{C}_n(z_\mathbb{C})$) the Zariski closure of $\mathcal{z}(\Lambda)$ in $\mathcal{X}$ (resp. of $\mathcal{z}_\mathbb{C}(\Lambda_\mathbb{C})$ in $\mathcal{X}$).  Then
\begin{align*}
h_{X,L}(f^n(z)) & = \left(\mathcal{C}_0(z)\cdot c_1(\mathcal{L})\cdot c_1(\pi^*\mathcal{N})^{\dim\mathcal{B}-1}\right)\\
& = \left(\mathcal{C}_0(z_\mathbb{C})\cdot c_1(\mathcal{L}_\mathbb{C})\cdot c_1(\pi^*\mathcal{N}_\mathbb{C})^{\dim\mathcal{B}_\mathbb{C}-1}\right)\\
& = h_{X_\mathbb{C},L_\mathbb{C}}(f_\mathbb{C}^n(z_\mathbb{C})),
\end{align*}
since the intersection numbers are invariant under extension of the field of constants.
We thus find $\widehat{h}_{f}(z)=\widehat{h}_{f_\mathbb{C}}(z_\mathbb{C})$.

We now can assume $\mathcal{B}$ is a complex variety and $\mathbf{K}:=\mathbb{C}(\mathcal{B})$. Let $(X,f,L)$ be a polarized endomorphism over $\mathbf{K}$ and let $(\mathcal{X},\mathcal{f},\mathcal{L})$ be a model of $(X,f,L)$ over $\mathcal{B}$. Applying Theorem~\ref{tm:mainPk} gives the first point, since $X$ is isomorphic to the generic fiber of $\pi:\mathcal{X}\to\mathcal{B}$, $L=\mathcal{L}|_X$ and $f=\mathcal{f}|_X$. For the second point, we can remark that, up to changing model, we can assume a point $z\in X(\mathbf{K})$ induces a marked point $a:\mathcal{B}\dashrightarrow \mathcal{X}$.
\end{proof}

\begin{example}[The elementary Desboves family II]\label{Desboves2}\normalfont
We keep the notations of Example~\ref{Desboves1}.
The critical set of $f_\lambda$ is the union of three lines $L_j$ passing through $\rho_0=[0:1:0]$ and $[1:0:\alpha^j]$ with $\alpha^3=1$, for $j=0,1,2$, each of them counted with multiplicity 2 together with the curve
\[\mathcal{C}_\lambda':=\{[x:y:z]\in\mathbb{P}^2\, : \ -x^3+z^3+\lambda(4y^3+x^3+z^3)=0\}.\]
Note that the lines $L_j$ correspond to preimages of critical points of the Latt\`es map $g:=f_\lambda|_Y$ by the fibration $p:\mathbb{P}^2\dasharrow Y$ which semi-conjugates $f_\lambda$ to $g$ and whose fibers are the lines of the pencil $\mathcal{P}$.

We now make a base change for our family : Let $\mathbf{K}\supset \C(\lambda)$ be the splitting field of the equation $w^3=(1+\lambda)/(1-\lambda)$. Then $\mathbf{K}=\C(\mathcal{B})$ for some smooth complex projective curve $\mathcal{B}$ and there exists a finite branched cover $\tau:\mathcal{B}\to\mathbb{P}^1$ of degree $[\mathbf{K}:\C(\lambda)]=3$. Changing base by $\tau$ allows to define three marked points $a_1,a_2,a_3:\mathcal{B}\to\mathbb{P}^2$ where $a_i(\lambda)=[x_i(\lambda):0:z_i(\lambda)]$ with $(1+\lambda)z_i(\lambda)^3+(1-\lambda)x_i(\lambda)^3=0$ (they parametrize intersection points of $H_y=\{y=0\}$ with $\mathcal{C}_\lambda'$).
We have defined a family $(\mathbb{P}^2\times\mathcal{B},\mathcal{f}_\mathcal{B},\mathcal{O}_{\mathbb{P}^2}(1))$ with regular part $\Lambda:=\tau^{-1}(\C^*)$. Let $\mathcal{Y}\subset\mathbb{P}^2\times\mathcal{B}$ be the subvariety such that $\mathcal{f}_\mathcal{B}(\mathcal{Y})=\mathcal{Y}$ given by Theorem~\ref{tm:mainPk}. As seen in Section~\ref{sec:algebraic},
\[H_y\times\mathcal{B} \subset\mathcal{Y} \quad \text{and} \quad \mathcal{C}\times\mathcal{B}\subset\mathcal{Y},\]
but $H_x\times\mathcal{B} \not\subset\mathcal{Y}$ and $H_z\times\mathcal{B} \not\subset\mathcal{Y}$. The subvariety $\mathcal{Y}$ may have other irreducible components, which would have to have period at least $2$ under iteration of $\mathcal{f}_\mathcal{B}$.

\medskip

Any constant marked point $\alpha\in H_y\cup\mathcal{C}$ is stable (but can have infinite orbit). Note also that $a_i(\lambda)\in\mathcal{Y}_\lambda$ for any $\lambda\in\C^*$. However $a_i$ is \emph{not} stable. Indeed, one has
\[\widehat{h}_{f}(a_i)=\int_{\mathbb{P}^2\times\Lambda}[\Gamma_{a_i}]\wedge\widehat{T}_{\mathcal{f}_\mathcal{B}}=\int_{\mathcal{B}}\tilde{a_i}^*(\mu_g)=\deg(\tilde{a_i})>0,\]
where $\tilde{a}_i(\lambda)=[x_i(\lambda):z_i(\lambda)]$ for all $\lambda\in\C^*$ and $\mu_g$ is the Green measure of $g:H_y\to H_y$ which gives no mass to proper analytic sets.
This example emphasizes that the conclusion that $\Gamma_{a_i}\subset\mathcal{Y}$ is not sufficient to have stability.
\end{example}

\section{Applications}\label{section5}

\subsection{The case of abelian varieties}

If $A$ is an abelian variety defined over $\mathbf{K}$, $L$ a symmetric ample line bundle on $A$ and $[n]$ the morphism of multiplication by the integer $n$ on $A$, the N\'eron-Tate height function $\widehat{h}_A:A(\bar{\mathbf{K}})\to\R$ of $A$ is defined by
\[\widehat{h}_A=\lim_{n\to\infty}\frac{1}{n^2}h_{A,L}\circ[n].\]
In particular, $\widehat{h}_A\circ [n]=n^2\widehat{h}_A$ for  any $n\in\mathbb{Z}$. Now fix $n\in\mathbb{Z}$ with $|n|\geq2$. Applying Corollary~\ref{cor:Northcott_function_field} to the polarized endomorphism $(A,[n],L)$ gives the next result due to Lang and N\'eron~\cite{Lang-Neron}:
\begin{theorem}[Lang-N\'eron]\label{tm:LangNeron}
Let $\mathbb{k}$ be a field of characteristic zero and let $\mathbf{K}$ be the function field of a normal $\mathbf{k}$-variety. Let $A$ be a non-isotrivial abelian variety defined over $\mathbf{K}$, $L$ be a symmetric ample line bundle on $A$ and $\tau:A_0\to A$ be the $\bar{\mathbf{K}}/\mathbf{k}$-trace of $A$. Then, there exists an integer $m$ such that for any $z\in A(\mathbf{K})$ with $\hat{h}_{A}(z)=0$, we have $[m]z\in \tau(A_0(\mathbf{K}))$.
\end{theorem}

\begin{proof}
As in the proof of Corollary~\ref{cor:Northcott_function_field}, we can assume $\mathbf{k}=\C$ without loss of generality. We apply Corollary~\ref{cor:Northcott_function_field} to the polarized endomorphism $(A,[2],L)$. We have a possibly reducible subvariety $Y=Y_1\cup \cdots \cup Y_M$ where $Y_i$ is periodic for all $i$, i.e.\ there exist integers $n_j>0$ with $[2^{n_j}]Y_j=Y_j$, for all $1\leq j\leq M$, which is  isotrivial. 
Up to taking a finite extension of $\mathbf{K}$, we can assume $[2^{n_j}]$ has a fixed point $p_j$ in $Y_j$. In particular, $p_j$ is a torsion point of $A$ and, if $A_j=Y_j-p_j$, then $A_j$ is a subgroup of $A$, whence $A_j$ is an abelian subvariety of $A$.

Note that, since $A_j$ is isotrivial, there exist an abelian variety $C_j$ defined over $\C$ and a morphism $\tau_j:C_j\to A$ such that $A_j=\tau_j(C_j)$. In particular, $A_j\subset \tau(A_0)$, which gives the sought statement.
\end{proof}

\subsection{A conjecture of Kawaguchi and Silverman}

Let $X$ be a projective variety defined over a field $\mathbf{K}$ of characteristic zero, which is either a number field or the function field of a projective variety defined over a field of characteristic zero, and $f:X\dashrightarrow X$ a dominant rational map. Recall the dynamical degrees $\lambda_j(f)$ are defined in \S\ref{dynamics_polarized}.
Let $X_f(\bar{\mathbf{K}})$ be the set of points whose full forward orbit is well-defined. In this case, one can also introduce the \emph{arithmetic degree} of a point $P\in X_f(\bar{\mathbf{K}})$ as  
\[\alpha_f(P):=\lim_{n\to+\infty}\max\left(h_X(f^n(P)),1\right)^{1/n},\]
where $h_X$ is any given Weil height function $h_X:X(\bar{\mathbf{K}})\to\mathbb{R}$, when the limit exists.

The following consequence of Theorem~\ref{tm:Pk} proves \cite[Conjecture 6]{Kawaguchi-Silverman-arithmetic-degree} in the case of polarized endomorphisms over function fields, see \cite{Kawaguchi-Silverman-arithmetic-degree} for the case of number fields, as well as ~\cite{Kawaguchi-Silverman-blocks,Lesieutre-Satriano,Matsuzawa-Sano-Shibata,Shibata,Matsuzawa}.
\begin{theorem}\label{tm:KS}
Let $\mathbf{k}$ be a field of characteristic zero and $\mathcal{B}$ be a normal projective $\mathbf{k}$-variety. Let $(X,f,L)$ be a non-isotrivial polarized endomorphism over $\mathbf{K}:=\mathbf{k}(\mathcal{B})$ of degree $d$. Fix a point $P\in X(\bar{\mathbf{K}})$. Then
\begin{enumerate}
\item the limit $\alpha_f(P)$ exists and is equal to $1$ or $d$,
\item if $P$ has Zariski dense orbit, we have $\alpha_f(P)=\lambda_1(f)=d$.
\end{enumerate}
\end{theorem}

\begin{proof}
Let us first remark that, since $(X,f,L)$ is polarized, we have $f^*L\simeq L^{\otimes d}$, so that
\[\left((f^n)^*c_1(L)\cdot c_1(L)^{\dim X-1}\right)=\left(c_1(L)^{{d^n}}\cdot c_1(L)^{\dim X-1}\right)=d^n\left( c_1(L)^{\dim X}\right),\]
which gives $\lambda_1(f)=d$.

We now take $\mathsf{x}\in X(\bar{\mathbf{K}})$. 
Let $\widehat{h}_f:X(\bar{\mathbf{K}})\to\R_+$ be the canonical height function of $f$. Assume $\widehat{h}_f(\mathsf{x})>0$. Then
\[h_{X,L}(f^n(\mathsf{x}))=\widehat{h}_f(f^n(\mathsf{x}))+O(1)=d^n\widehat{h}_f(\mathsf{x})+O(1),\]
so that $\alpha_f(\mathsf{x})=d$, as sought. Note also that if $\widehat{h}_{f}(\mathsf{x})=0$, this implies
\[h_{X,L}(f^n(\mathsf{x}))=O(1),\]
so that $\alpha_f(\mathsf{x})=1$.
We are thus left with proving that, if $\mathsf{x}$ has a Zariski dense orbit, then $\widehat{h}_{f}(\mathsf{x})>0$.
Assume $\mathsf{x}$ has Zariski dense orbit, but $\widehat{h}_f(\mathsf{x})=0$. Let $\mathbf{K}'$ be the field of rational functions of some normal projective variety $\mathcal{B}'$ such that $\mathsf{x}$ is defined over $\mathbf{K}'$. Let $(\mathcal{X},\mathcal{f},\mathcal{L})$ be a non-isotrivial model of $(X,f,L)$, $\pi:\mathcal{X}\to\mathcal{B}'$ be the projection onto $\mathcal{B}'$, and $\Lambda$ be a regular part. Up to changing model, to $\mathsf{x}$, we can associate a marked point $x:\mathcal{B}'\dashrightarrow\mathcal{X}$ and a subvariety $\mathcal{C}\subset\mathcal{X}$ (which is the Zariski closure of $x(\Lambda)$ in $\mathcal{X}$ and which is flat over $\Lambda$). According to Theorem~\ref{tm:caracterization-stable}, we have
\[\sup_n\deg_\mathcal{M}(\mathcal{f}^n(\mathcal{C}))=D<+\infty.\]
We now can apply Theorem~\ref{tm:Pk}, and we deduce the existence of $\mathcal{Y}_D$ a strict subvariety of $\mathcal{X}$, such that $\mathcal{f}^n(\mathcal{C})\subset\mathcal{Z}$ for all $n\geq N$. If $Y$ is the generic fiber of $\mathcal{Y}_D\to\mathcal{B}'$, we have $f^n(\mathsf{x})\in Y$ for any $n\geq N$, which is a contradiction.
\end{proof}

\par\noindent We refer the reader to \cite{DGHLS} for generalization of the conjecture of Kawaguchi and Silverman where Theorem \ref{tm:formulaheight} provides partial answers.

\begin{remark}\normalfont
If $(X,f,L)$ is isotrivial, it is possible that $X(\mathbf{k})$ is Zariski dense in $X$. In particular, there can be a point $x\in X(\mathbf{k})$ with Zariski dense orbit and $\alpha_f(x)=1$ for this point. In particular the second assertion of Theorem~\ref{tm:KS} cannot hold true in general, without assuming $(X,f,L)$ is non-isotrivial.
\end{remark}

\section{The geometric dynamical Bogomolov conjecture}\label{sec:Bogomolov}

In this section, we let $\mathbf{k}$ be a field of characteristic $0$, we also let $\mathcal{B}$ be a normal projective $\mathbf{k}$-variety and we let $\mathbf{K}:=\mathbf{k}(\mathcal{B})$ be the field of rational functions of $\mathcal{B}$.
\subsection{Motivation of the conjecture}
As mentioned in the introduction, the geometric Bogomolov conjecture has been proved in \cite{Cantat-Gao-Habegger-Xie}. Let us first state their result.
\begin{theorem}[\cite{Cantat-Gao-Habegger-Xie}]
Let $A$ be an abelian variety, $L$ be a symmetric ample line bundle on $A$ and $\widehat{h}_A$ be the N\'eron-Tate height of $A$ relative to $L$. Let $Z$ be an irreducible subvariety of $A_{\bar{\mathbf{K}}}$. Assume for all 
$\varepsilon>0$, the set $Z_\varepsilon:=\{x\in Z(\bar{\mathbf{K}})\, : \ \widehat{h}_A(x)<\varepsilon\}$
is Zariski dense in $Z$. Then, there exist a torsion point $a\in A(\bar{\mathbf{K}})$, an abelian subvariety $C\subset A$ and a subvariety $W\subset A_0$ of the $\bar{\mathbf{K}}/\mathbf{k}$-trace $(A_0,\tau)$ of $A$ such that
\[Z=a+C+\tau(W\otimes_k\bar{\mathbf{K}}).\]
\end{theorem}

Let us now give a dynamical formulation of the conclusion of this result.
 Fix $n\geq2$ and let $[n]:A\to A$ and $\tilde{[n]}:A_0\to A_0$ be the respective multiplication by $n$ morphisms. 
Let $M\geq0$ be such that $b:=[n]^M(a)$ is periodic under iteration of $[n]$, i.e.\  there exists $k\geq1$ with $[n]^kb=b$, and let
\[V:=b+C+\tau(A_0\otimes_\mathbf{k}\bar{\mathbf{K}}).\]

The variety $V$ is fixed by $[n]^k$ and, if $V_0:=\tau(A_0\otimes_\mathbf{k}\bar{\mathbf{K}})$, there is a fibration
$p:V\to V_0$
which is invariant by $[n]^k$, i.e.\  such that the following diagram commutes:
\[\xymatrix{
V\ar[r]^{[n]^k}\ar[d]_{p}  & \ar[d]^{p}V\\
V_0 \ar[r]_{[n]^k|_{V_0}} & V_0}\]
and such that
$(V_0,[n]^k|_{V_0},L|_{V_0})$ is isotrivial. We thus may rephrase the conclusion as follows: There exist subvarieties $V,V_0\subset A$ invariant by $[n]^k$ and an integer $N\geq1$ and a surjective morphism $p:V\to V_0$ such that
\begin{enumerate}
\item $p\circ([n]^k|_V)=([n]^k|_{V_0})\circ p$ and $(V_0,[n]|_{V_0},L|_{V_0})$ is isotrivial,
\item $[n]^N(Z)=p^{-1}(W_0)$ where $W_0$ is an isotrivial subvariety of $V_0$.
\end{enumerate}
  In particular, Conjecture~\ref{Conjecture_Bogomolov} in the Introduction is a generalization of this dynamical formulation of the geometric Bogomolov conjecture to the case where $(X,f,L)$ is a non-isotrivial polarized endomorphism defined over $\mathbf{K}$, $X$ is normal and $Z\subset X_{\bar{\mathbf{K}}}$ is irreducible subvariety.
  
  \medskip
  
Recall that in the Introduction, we defined a subvariety $Z\subseteq X$ as \emph{$f$-special} if there exist an integer $k$, a polarized endomorphism $(X,\Psi,L)$ a subvariety $Y$ with $Z\subseteq Y \subseteq X$ such that
\begin{itemize}
\item $f^k(Y)=Y=\Psi(Y)$;
\item $f^k\circ \Psi = \Psi \circ f^k$ on $Y$;
\item either $Z$ is preperiodic under $\Psi$ or $Z$ comes from an isotrivial factor of $(X,\Psi,L)$.
\end{itemize}
As explained in the Introduction, the existence of non-isotrivial families of CM abelian varieties in relative dimension at least $3$ imposes this definition. However, the situation is simpler when studying abelian varieties $A$ over function fields endowed with the map $[n]$: if $Z$ is an $[n]$-special subvariety, we can take $Y=X$ and $\Psi=[n]^k$.

\medskip

Recall that, by \cite{Faber,Gubler}, $\widehat{h}_f$ is induced by an $M_\mathcal{B}$-metric in the sense of Gubler \cite{Gubler-Bogomolov}, and that, in this case, if $Z\subset X_{\bar{\mathbf{K}}}$ is an irreducible subvariety, we have the Zhang inequalities which imply the following are equivalent by \cite[Corollary 4.4]{Gubler-Bogomolov}:
\begin{itemize}
\item the height of $Z$ is zero, i.e.\  $\widehat{h}_{f}(Z)=0$,
\item the essential minimum of the metrization vanishes:
\[e_1:=\sup_{Y}\inf_{x\in Z(\bar{\mathbf{K}})\setminus Y}\widehat{h}_f(x)=0,\]
where $Y$ ranges over all hypersurfaces of $Z$,
\item for any $\varepsilon>0$ the set $Z_\varepsilon=\{x\in Z(\bar{\mathbf{K}}); \ \widehat{h}_f(x)\leq \varepsilon\}$ is Zariski-dense in $Z$.
\end{itemize}
In particular, Conjecture~\ref{Conjecture_Bogomolov} is a characterization of the subvarieties $Z$ with $\widehat{h}_f(Z)=0$.

\subsection{Stable fibers of an invariant fibration}

We want to explore basic properties of subvarieties with height $0$. The first thing we want to do is to relate, when $f$ preserves a fibration, the height of the fiber over $y$ to the height of $y$. 

We let $\mathbf{K}=\C(\mathcal{B})$ be the field of rational function of a complex normal projective variety $\mathcal{B}$. When $X$ and $Y$ are projective varieties over $\mathbf{K}$ and $p:X\dashrightarrow Y$ is a dominant rational map defined over $\mathbf{K}$, for any irreducible subvariety $W\subset  Y_{\bar{\mathbf{K}}}$, we denote by $X_W$ the ``fiber'' $p^{-1}(W)$ of $p$ over $W$.

\begin{lemma}\label{lm:caseoffibers}
Let $(X,f,L)$ be a polarized endomorphism over $\mathbf{K}$ of degree $d>1$. Assume there exist a polarized endomorphism $(Y,g,E)$ over $\mathbf{K}$ of degree $d$ with $\dim Y<\dim X$, and a dominant rational map $p:X\dashrightarrow Y$ defined over $\mathbf{K}$ with $p\circ f=g\circ p$.
For any subvariety $W\subset Y_{\bar{\mathbf{K}}}$, we have 
\[\widehat{h}_{f}(X_W)=\widehat{h}_{g}(W).\]
In particular, $\widehat{h}_{f}(X_W)=0$ if and only if $\widehat{h}_{g}(W)=0$.
\end{lemma}

\begin{proof}
Let $(\mathcal{X},\mathcal{f},\mathcal{L})$ and $(\mathcal{Y},\mathcal{g},\mathcal{E})$ be respective models for $(X,f,L)$ and $(Y,g,E)$ with common regular part $\Lambda$, and let $\mathcal{p}:\mathcal{X}\dashrightarrow\mathcal{Y}$ be a model for $p:X\dashrightarrow Y$. For any subvariety $W\subset Y_{\bar{\mathbf{K}}}$, we let $\mathcal{W}$ be the Zariski closure of $W$ in $\mathcal{Y}$.
Recall that $X_W:=\mathcal{p}^{-1}(W)$. We set $\mathcal{X}_W:=\mathcal{p}^{-1}(\mathcal{W})$, and $r:=\dim X-\dim Y>0$, and $s:=\dim W$.
We remark that $\widehat{T}_\mathcal{g}^{s+1}=\mathcal{p}_*\left(\widehat{T}_\mathcal{f}^{r+s+1}\right)$, so that the projection formula and Theorem~\ref{tm:formulaheight} give
\begin{align*}
\widehat{h}_{g}(W) =\widehat{h}_{f}(X_W),
\end{align*}
where we used that $\pi_{\mathcal{X}}\circ \mathcal{p}=\pi_{\mathcal{Y}}$.
\end{proof}

We deduce the following.

\begin{proposition}
Let $(X,f,L)$ be a polarized endomorphism over $\mathbf{K}$ of degree $d>1$ and let $Z$ be a $f$-special subvariety. Then $\widehat{h}_f(Z)=0$.
\end{proposition}
\begin{proof}
As $\widehat{h}_f=\widehat{h}_{f^k}$, for any $k\geq1$, up to replacing $f$ by an iterate, we can suppose there is a polarized endomorphism $(X,\Psi,L)$ and a subvariety $Y$ with $Z\subseteq Y \subseteq X$ such that
\begin{itemize}
\item $f(Y)=Y=\Psi(Y)$;
\item $f\circ \Psi = \Psi \circ f$ on $Y$;
\item either $Z$ is preperiodic under $\Psi$ or $Z$ comes from an isotrivial factor of $(X,\Psi,L)$.
\end{itemize}
By Lemma~\ref{lm:caseoffibers}, we have $\widehat{h}_{\Psi}(Z)=0$. All there is left to prove is that $\widehat{h}_{f}(Z)=\widehat{h}_{\Psi}(Z)$. Let $(\mathcal{X},\mathcal{f},\mathcal{L})$ and $(\mathcal{X},\psi,\mathcal{L})$ be models of $(X,f,L)$ and $(X,\Psi,L)$ respectively. As $\mathcal{f}\circ \psi=\psi\circ \mathcal{f}$ over any common regular part $\Lambda$, we have
\[\mathcal{f}^*\psi^*\widehat{T}_f=d\psi^*\widehat{T}_f \quad \text{on} \ \mathcal{Y}_\Lambda,\]
so that $\widehat{T}_\psi=\widehat{T}_f$ on $\mathcal{Y}_\Lambda$ (by uniqueness of the backward invariant current with locally bounded potentials). Let $\omega_\mathcal{B}$ be a K\"ahler form on $\mathcal{B}$ which represents $\mathcal{N}$ and let $\mathcal{Z}$ be the Zariski closure of $Z$ in $\mathcal{Y}$. Over any regular part $\Lambda$ over which $\mathcal{Z}$ is flat, we thus have
\begin{align*}
\int_{\mathcal{Y}_\Lambda}\widehat{T}_\psi^{\dim Z_\eta+1}\wedge[\mathcal{Z}]\wedge \pi^*(\omega_\mathcal{B})^{\dim\mathcal{B}-1} & =\int_{\mathcal{Y}_\Lambda}\widehat{T}_f^{\dim Z_\eta+1}\wedge[\mathcal{Z}]\wedge \pi^*(\omega_\mathcal{B})^{\dim\mathcal{B}-1}\\
& =\int_{\mathcal{X}_\Lambda}\widehat{T}_f^{\dim Z_\eta+1}\wedge[\mathcal{Z}]\wedge \pi^*(\omega_\mathcal{B})^{\dim\mathcal{B}-1},
\end{align*}
where we used that $\mathcal{Z}\subset\mathcal{Y}\subset\mathcal{X}$. The conclusion follows from Theorem~\ref{tm:formulaheight}.
\end{proof}

Consider a trivial family $(g(z),\lambda)$ on $\p^1\times \Lambda$, with $\deg(g)=d\geq2$ and a marked point $a:\Lambda \to \p^1$. Then, it is easy to see that the graph of $a$ is stable if and only if $a$ is constant (if not, $a(\lambda)$ is a repelling periodic point of $g$ at some $\lambda$). More generally, when $p:X\dashrightarrow\p^1$ and $g:\p^1\to\p^1$ is isotrivial, there exists an automorphism $\phi:\p^1\to\p^1$ defined over $\bar{\mathbf{K}}$
such that $\phi^{-1}\circ g\circ \phi $ is defined over $\mathbb{C}$ and we say that $X_z$ is an \emph{isotrivial fiber} of $p$ if $X_z=p^{-1}\{z\}$ where $\phi^{-1}(z)$ is defined over $\mathbb{C}$.

As a particular case of Lemma~\ref{lm:caseoffibers}, we have 

\begin{corollary}\label{cor:fiber-overP1}
Let $(X,f,L)$ be a non-isotrivial polarized endomorphism over $\mathbf{K}$ of degree $d>1$. Assume there exist a polarized endomorphism $(\p^1,g,\mathcal{O}_{\p^1}(1))$ over $\mathbf{K}$ of degree $d$, and a dominant rational map $p:X\dashrightarrow \p^1$ defined over $\mathbf{K}$ with $p\circ f=g\circ p$. For any $y\in \p^1(\mathbf{K})$, let $X_y$ be the fiber $p^{-1}\{y\}$ of $p$.
\begin{enumerate}
\item If $(\p^1,g,\mathcal{O}_{\p^1}(1))$ is non-isotrivial, then $\widehat{h}_{f}(X_y)=0$ if and only if $X_y$ is preperiodic under iteration of $f$,
\item if $(\p^1,g,\mathcal{O}_{\p^1}(1))$ is isotrivial, then $\widehat{h}_f(X_y)=0$ if and only if $X_y$ is an isotrivial fiber of $p$.
\end{enumerate}
\end{corollary}

\begin{proof}
Assume first $g$ is non-isotrivial. According to Lemma~\ref{lm:caseoffibers}, we have $\widehat{h}_g(X_y)=0$ if and only if $\widehat{h}_g(y)=0$. Since $(\p^1,g,E)$ is non-isotrivial, $\widehat{h}_g(y)=0$ if and only if $y$ is $g$-preperiodic.
As $f(X_y)=X_{g(y)}$, the variety $X_y$ has to be preperiodic.

Assume now $g$ is isotrivial. Then there exist a finite extension $\mathbf{K}'$ of $\mathbf{K}$ and an affine automorphism $\phi:\mathbb{P}^1\to\mathbb{P}^1$, defined over $\mathbf{K}'$ such that $g_0:=\phi^{-1}\circ g\circ \phi$ is defined over $\C$. Let $\rho:\mathcal{B}'\to\mathcal{B}$ be a finite branched cover with $\mathbf{K}'=\C(\mathcal{B}')$ and let $(\p^1\times\mathcal{B}',\mathcal{g},\mathcal{O}_{\p^1}(1))$ be a model for $(\p^1,g,\mathcal{O}_{\p^1}(1))$ over $\mathcal{B}'$. Let $\sigma:\mathcal{B}'\dashrightarrow\p^1$ be the rational map induced by $\phi^{-1}(y)$ and $\mathcal{Y}$ be the Zariski closure of $y$ in $\p^1\times\mathcal{B}'$. If $\Phi:\p^1\times\mathcal{B}'\dashrightarrow \p^1\times\mathcal{B}'$ is a model of $\phi$ and $\Lambda$ is a common regular part for all above models, then
\begin{align*}
\widehat{h}_{g}(y) & =\int_{\p^1\times \Lambda}[\mathcal{Y}]\wedge \widehat{T}_\mathcal{g}\wedge(\pi_{\mathcal{B'}}^*\omega_{\mathcal{B}'})^{\dim\mathcal{B}'-1}  = \int_{\p^1\times \Lambda}[\mathcal{Y}]\wedge \Phi_*\left(\pi_{\p^1}^*(\mu_{g_0})\right)\wedge(\pi_{\mathcal{B'}}^*\omega_{\mathcal{B}'})^{\dim\mathcal{B}'-1}\\
& = \int_{\p^1\times \Lambda}[\Gamma_\sigma]\wedge\pi_{\p^1}^*(\mu_{g_0})\wedge(\pi_{\mathcal{B'}}^*\omega_{\mathcal{B}'})^{\dim\mathcal{B}'-1} = \int_{\Lambda}\sigma^*(\mu_{g_0})\wedge\omega_{\mathcal{B}'}^{\dim\mathcal{B}'-1}=\deg_{\mathcal{N}}(\sigma).
\end{align*}
In particular, $\widehat{h}_{g}(y)=0$ if and only if $\sigma$ is constant, i.e.\  $\phi^{-1}(y)\in \C$. This concludes the proof.
\end{proof}

\subsection*{Acknowledgment}
The first author would like to heartily thank S\'ebastien Boucksom for useful comments and discussions about families of varieties and Damian Brotbek for discussions which led to the formulation of Theorem~\ref{tm:Pk_intro} in the context of polarized endomorphisms of normal projective varieties. We would like to thank Matthieu Astorg for interesting discussions at an early stage of this work. We also would like to thank Antoine Chambert-Loir and Charles Favre. We are indebted to Najmuddin Fakhruddin for many interesting discussions, notably for sharing the references \cite{Rapoport, Ghioca_Tucker} with us.
Finally, we would like to thank the anonymous referees for their very interesting comments and suggestions which helped to improve the presentation of this work.

\bibliographystyle{short}
\bibliography{biblio}

\end{document}